\documentclass[12pt]{amsart}
\usepackage{amsmath}
\usepackage{amsthm}
\usepackage{amssymb}
\usepackage{enumerate}
\usepackage{xcolor}
\usepackage{comment}

\newtheorem{theorem}{Theorem}[section]
\newtheorem{lemma}[theorem]{Lemma}
\newtheorem{proposition}[theorem]{Proposition}
\newtheorem{corollary}[theorem]{Corollary}

\newtheorem{remark}[theorem]{Remark}

\newcommand{\bfind}[1]{\index{#1}{\bf #1}}
\newcommand{\n}{\par\noindent}

\newcommand{\sn}{\par\smallskip\noindent}
\newcommand{\mn}{\par\medskip\noindent}
\newcommand{\bn}{\par\bigskip\noindent}
\newcommand{\pars}{\par\smallskip}
\newcommand{\parm}{\par\medskip}
\newcommand{\parb}{\par\bigskip}

\newcommand{\cG}{{\mathcal G}}
\newcommand{\ovl}[1]{\overline{#1}}
\newcommand{\fs}{^{\uparrow}}
\newcommand{\MX}{_{\rm max}}
\newcommand{\mfs}{-_{\rm mc}}
\newcommand{\ms}{-_{\rm fs}}
\newcommand{\ann}{\mbox{\rm ann}\,}
\newcommand{\Delt}{\Delta}
\newcommand{\dhat}{^\diamondsuit}

\newcommand{\cO}{\mathcal{O}}
\newcommand{\cM}{\mathcal{M}}
\newcommand{\cR}{\mathcal{R}}

\newcommand{\Z}{\mathbb Z}
\newcommand{\N}{\mathbb N}
\newcommand{\Q}{\mathbb Q}

\begin{document}

\title[Cuts and ideals over valuation rings]{Arithmetic of cuts in ordered abelian groups 
and of ideals over valuation rings}
\author{Franz-Viktor and Katarzyna Kuhlmann}
\address{Institute of Mathematics, University of Szczecin,	
ul.\ Wielkopolska 15, 70-451 Szczecin, Poland}
\email{fvk@usz.edu.pl}
\email{katka314@gmail.com}

\thanks{We thank the referee for his many very helpful corrections and comments.}


\date{6.\ 10.\ 2025}

\begin{abstract}
We investigate existence, uniqueness and maximality of solutions $T$ for equations 
$S_1+T=S_2$ and inequalities
$S_1+T\subseteq S_2$ where $S_1$ and $S_2$ are final segments of ordered abelian groups.
Since cuts are determined by their upper cut sets, which are final segments, this gives
information about the corresponding equalities and inequalities for cuts. We apply our
results to investigate existence, uniqueness and maximality of solutions $J$ for equations
$I_1 J=I_2$ and inequalities $I_1 J\subseteq I_2$ where $I_1$ and $I_2$ are ideals of
valuation rings. This enables us to compute the annihilators of quotients of the form 
$I_1/I_2\,$.

\end{abstract}

\maketitle

%
%
\section{Introduction}
For the notions and notation we use in this paper, see Sections~\ref{sectprelfs},  
\ref{sectfsoag}, \ref{sectprelid} and~\ref{sectivr}. 
One main purpose of this manuscript is to address a problem that comes up in the papers
\cite{pr1,pr2}. There we work with valued field extensions $(L|K,v)$ of prime degree with
unique extension of the valuation $v$ from $K$ to $L$. We show that the module of
relative differentials (Kähler differentials) for the extension $\cO_L|\cO_K$ of the
respective valuation rings is isomorphic, as $\cO_L$-module, to a quotient $U/UV$ of ideals $U$ and $V$ of $\cO_L\,$. We wish to compute the annihilator of
such a quotient $U/UV$. It will contain $V$, but there are cases where it is bigger.

In the present manuscript, we will consider the more general case of quotients $I_1/I_2$ of
$\cO_v$-ideals $I_1,I_2$ (see Section~\ref{sectann}). To this end, we consider equations 
$I_1 J=I_2$ and inequalities $I_1 J\subseteq I_2\,$, as we will describe in detail below.

Our approach is to translate questions about ideals to corresponding questions about
final segments in
the value group $vK$ of $(K,v)$. The latter are connected with the investigation of cuts in 
ordered abelian groups which has appeared at several points in the literature. For a survey 
and a list of references, see \cite{Ku60}. The present paper uses, and improves, parts of 
the unpublished manuscript \cite{pr3}. Let us describe the background on cuts first.

Take an ordered abelian group $(G,\leq)$. By a \bfind{cut} in $(G,\leq)$ we mean a pair 
\[
\Lambda\>=\>(\Lambda^L,\Lambda^R)\>,
\]
where $\Lambda^L$ is an initial segment of $G$ and $\Lambda^R$ is a final segment of $G$
such that $\Lambda^L\cup\Lambda^R=G$ and $\Lambda^L\cap\Lambda^R=\emptyset$. We call 
$\Lambda^L$ the \bfind{lower cut set} of $\Lambda$, and $\Lambda^R$ the \bfind{upper cut 
set} of $\Lambda$. For basic information on cuts and final segments, see
Sections~\ref{sectprelfs} and~\ref{sectfsoag}. {\it Note that when we speak of final 
segments we will always assume that they are proper nonempty subsets of $G$.}

In the present paper we are particularly interested in the \bfind{addition of
cuts}. The two immediately obvious ways to define the sum $\Lambda_1+\Lambda_2$ of two
cuts $\Lambda_1$ and $\Lambda_2$ are the following:
\sn
1) \ the lower cut set of $\Lambda_1+\Lambda_2$ is $\Lambda_1^L+\Lambda_2^L$,
\n
2) \ the upper cut set of $\Lambda_1+\Lambda_2$ is $\Lambda_1^R+\Lambda_2^R$,
\sn
where for any $S_1,S_2\subseteq G$, $S_1+S_2:=\{\alpha+\beta\mid \alpha\in S_1\,,\, \beta
\in S_2\}$. The two additions usually do not yield the same results, but their properties 
are very similar. Under both additions, the set of all cuts in $G$ is a commutative monoid.
However, in general it is not a group as for fixed $\Lambda$ the function
\[
\Lambda'\>\mapsto\> \Lambda+\Lambda' 
\]
may not be injective. There may even be idempotents; see Section~\ref{sectfsoag}. 

In the present paper, we will exclusively work with the second definition, which we 
call the \bfind{upper cut set addition}. Every cut is uniquely determined by its upper cut 
set, so instead of cuts we will just work with final segments (for a reason that will be
explained later). The set of final segments of $G$ is linearly ordered by inclusion, and 
unions, intersections and sums of any collections of final segments are again final
segments, as long as they are proper nonempty subsets of $G$.

\pars
From our observations about cut addition, the following questions arise:
\sn
{\it Take final segments $S_1$ and $S_2$ of $G$. 
\sn
(QFS1) \ Is there a final segment $T$ of $G$ such that
\begin{equation}                            \label{eqT}
S_1\,+\,T\>=\>S_2 
\end{equation} 
holds? 
\n
(QFS2) \ If yes, is $T$ uniquely determined? If so, determine it.
\n
(QFS3) \ If $T$ exists but is not uniquely determined, determine the largest final segment
$T\MX$ such that $S_1+T\MX = S_2\,$. 
\n
(QFS4) \ If $T$ does not exist, compute the largest final segment $T\MX$ such that $S_1 +
T\MX\subset S_2\,$.} 

One aim of this paper is to answer these questions, which will be done in 
Section~\ref{sectfseq}. Obviously,
\begin{equation}                                \label{-}
S_2\,\ms\, S_1\>:=\>\{\alpha\in G\mid \alpha+S_1\subseteq S_2\}
\end{equation}
is the largest solution $T\MX$ for the inequality $S_1 + T\subseteq S_2\,$; hence if a
solution of (\ref{eqT}) exists, then $S_2\ms S_1$ is the largest. However, in the
literature there does not seem to exist any recipe for computing $S_2\ms S_1$ in terms of
(elementwise) sums and differences of final segments derived from $S_1$ and $S_2$ that
would help solve the problems that came up in the papers \cite{pr1,pr2}
and which we will explain below. 

A key to the answers to the above questions is the notion of {\it invariance group} of a 
final segment, which always is a convex subgroup of $G$ (see Section~\ref{sectig}). This 
sheds light on the important role that convex subgroups play in nonarchimedean ordered 
groups for the answers to the above questions. The case of archimedean ordered groups, 
which have no proper nontrivial convex subgroups, is significantly easier. Our answers to 
the above questions are given in Theorems~\ref{solv} and~\ref{maxsol3}.

Our interest in these questions arose from our work with cuts and its applications 
to ordered abelian groups and ordered fields, see \cite{Ku60,pr3} and the citations therein.
Recently, it gained even more importance for us because of applications to the theory of 
ideals over valuation rings.

Take a valued field $(K,v)$ with value group $vK$ and valuation ring $\cO_v$. For $a\in K$ 
we write $va$ in place of $v(a)$. When we talk of \bfind{$\cO_v$-ideals} we will always 
assume them to be nonzero, and we also include fractional ideals; in other words, {\it we 
talk about .} Just like
the set of final segments of $vK$, the set of $\cO_v$-ideals is linearly ordered by 
inclusion, and unions, intersections and products of any collections of $\cO_v$-ideals are
again $\cO_v$-ideals, as long as they are nonzero $\cO_v$-modules properly contained
in $K$. The function
\begin{equation}                           \label{vIS}
v:\> I\>\mapsto\> vI\>:=\> \{vb\mid 0\ne b\in I\}
\end{equation}
is an order preserving bijection from the set of all $\cO_v$-ideals onto the set of all
final segments of $vK$: $J\subseteq I$ holds
if and only if $vJ\subseteq vI$ holds. The inverse of this function 
is the order preserving function
\begin{equation}                           \label{S->IS}
S\>\mapsto\> I_S\>:=\> (a\in K\mid va\in S)\>=\>\{a\in K\mid va\in S\}\cup\{0\}\>.
\end{equation}
Further, the function (\ref{vIS}) is a homomorphism from the multiplicative monoid of 
$\cO_v$-ideals onto the additive monoid of final segments:
\begin{equation}                           \label{vIS+}
vIJ\>=\> vI+vJ\;\;\mbox{ and }\;\; I_{S+S'}\>=\>I_S I_{S'}\>.
\end{equation}

Via the function (\ref{vIS}), the following questions translate to questions (QFS1)-(QFS4): 
\sn
{\it Take nonzero $\cO_v$-ideals $I_1$ and $I_2\,$.
\sn
(QID1) \ Is there an $\cO_v$-ideal $J$ such that            
\begin{equation}                       \label{eqid}
I_1\,J\>=\>I_2 
\end{equation} 
holds? 
\n
(QID2) \ If yes, is $J$ uniquely determined? If so, compute it.
\n
(QID3) \ If $J$ exists but is not uniquely determined, compute the largest $\cO_v$-ideal 
$J\MX$ such that $I_1 J\MX=I_2\,$.
\n
(QID4) \ If $J$ does not exist, compute the largest $\cO_v$-ideal $J\MX$ such that 
$I_1 J\MX\subset I_2\,$.} 

We will derive answers to these questions in Section~\ref{sectideq}, using the results of 
Section~\ref{sectfseq}. Here, a key role is played by the notion of the {\it invariance
valuation ring} of an ideal $I$, which corresponds to the invariance group of the final 
segment~$vI$.

Obviously,
\begin{equation}                                \label{:}
I_2: I_1\>:=\>\{a\in K\mid aI_1\subseteq I_2\}
\end{equation}
is the largest solution $J\MX$ for the inequality $I_1J\subseteq I_2\,$. Again, in the
literature there does not seem to exist any recipe for computing $I_1:I_2$ from
$I_1$ and $I_2$ that would help solve the problems addressed in Section~\ref{sectann}. 
This section is devoted to the computation of the annihilators $I_2:I_1$ of 
quotients $I_1/I_2$ of $\cO_v$-ideals $I_1,I_2$ with $I_2\subseteq I_1\,$. The results
are then applied in Section~\ref{sectannsp} to compute the annihilators of 
some special quotients that appear in \cite{pr1} and \cite{pr2} and to answer the question 
in which cases they are equal to the maximal ideal $\cM_v$ of $\cO_v\,$. Our answers to the
above questions are given in Theorems~\ref{idsolv2} and~\ref{idsolv3}.

\pars
Ideals over valuation domains have been studied extensively in the literature, see for 
instance \cite[Chapter II, \S4]{FS}. This article lists a number of properties of such 
ideals; in 
Section~\ref{sectM(I)} we present the corresponding properties of final segments of the 
value group and use them to give alternate proofs for the properties of ideals $I$ by
deriving them from the properties of the final segments $vI$.

Equations of the form $I_1 J=I_2$ are implicitly related to groups in the semigroups 
formed by the isomorphy classes of the ideals, which for instance are studied in \cite{BS}. 
Further, the connection between ideals of valuation rings and the final segments of value 
groups is described and used in \cite{BS}, and it is pointed out that something similar
was already done by Paulo Ribenboim in \cite{R}. In fact, Ribenboim arrives at the same
criteria that answer questions (QFS1) and (QID1) as we give in Theorems~\ref{solv}
and~\ref{idsolv2}.
Also invariance groups appear in \cite{R}, in a somewhat disguised form. However, the
machinery used by Ribenboim and cited in \cite{BS} is heavy. As it turns out in the present
paper, embedding the value group in a Hahn product and even working with the natural 
valuation of an ordered abelian group is not necessary. The fact that final segments in 
ordered abelian groups of higher rank may not have infima in suitably defined extensions of 
the group may be scary, but it really is no obstacle. 


\mn
%
%
%
\section{Final segments}                   \label{fs}
%

%
%
%
\subsection{Preliminaries on cuts and final segments}          \label{sectprelfs}
\mbox{ }\sn
Take $(\Gamma,\leq)$ to be the underlying ordered set of any nonzero ordered abelian group
({\it by ``ordered'', we will always mean ``totally ordered''}). Note that $(\Gamma,\leq)$
does not have a smallest or largest element, and the definitions in this section also work
without the presence of the group structure. If $M_1, M_2$ are nonempty subsets of
$\Gamma$ and
$\alpha\in \Gamma$, we will write $\alpha<M_2$ if $\alpha<\beta$ for all $\beta\in M_2$, 
and we will write $M_1<M_2$ if $\alpha<M_2$ for all $\alpha\in M_1$.
Similarly, we use the relations $>$, $\leq$ and $\geq$ in place of $<$.

An ordered set $(\Gamma,\leq)$ is \bfind{discretely ordered} if every element has an 
immediate successor, and it is \bfind{densely ordered} if for every $\alpha,\beta\in 
\Gamma$ with $\alpha<\beta$ there is $\gamma\in\Gamma$ with $\alpha<\gamma<\beta$.

A subset $M$ of $\Gamma$ is called \bfind{convex} {\bf in} $(\Gamma,\leq)$ if for
every two elements $\alpha,\beta\in M$ and every $\gamma\in \Gamma$ such that $\alpha\leq 
\gamma\leq\beta$, it follows that $\gamma\in M$. A subset $S$ of $\Gamma$ is an 
\bfind{initial segment} {\bf of} $\Gamma$ if for every $\alpha\in S$ and every $\gamma\in
\Gamma$ with $\gamma \leq \alpha$, it follows that $\gamma\in S$. Symmetrically, $S$ is a
\bfind{final segment} {\bf of} $\Gamma$ if for every $\alpha \in S$ and every $\gamma\in 
\Gamma$ with $\gamma\geq \alpha$, it follows that $\gamma\in S$. {\it Throughout, 
we will assume initial and final segments to be nonempty and not equal to $\Gamma$.} Then
$S$ is a final segment of $\Gamma$ if and only if $\Gamma\setminus S$ is an initial
segment of $\Gamma$. We will denote the set of all final segments of $(\Gamma,\leq)$ by
\[
\Gamma\fs\>. 
\]

If $\emptyset\ne M_1\,,\,M_2\subsetneq \Gamma$ are such that $M_1\leq M_2$ and $\Gamma=M_1
\cup M_2$, then we call $(M_1,M_2)$ a \bfind{quasi-cut} in $\Gamma$, and we write
$\Lambda^L=M_1\,$, $\Lambda^R=M_2\,$, and $\Lambda=(\Lambda^L,\Lambda^R)$. It follows that
$M_1$ is an initial segment of $\Gamma$, $M_2$ is a final segment of $\Gamma$, and the
intersection of $M_1$ and $M_2$ consists of at most one element. If this intersection is
empty, then $(M_1,M_2)$ is called a \bfind{(Dedekind) cut} in $\Gamma$.

\pars
For any nonempty subset $M\subseteq \Gamma$, we set
\[
M^+\>:=\>\{\gamma\in \Gamma\mid \gamma>M\}\>.
\]
If this is nonempty, it is the largest final segment having empty intersection with $M$.
Similarly, we set
\[
M^-\>:=\>\{\gamma\in \Gamma\mid \exists m\in M: \gamma\geq m\}\>.
\]
That is, if $M^-\ne\Gamma$, then it  is the smallest final segment of $\Gamma$ which
contains $M$.

We will write $\gamma^+$ instead of $\{\gamma\}^+$ and $\gamma^-$ instead of $\{\gamma\}^-$. 
A final segment of the form $\gamma^-$ will be called \bfind{principal}; it has smallest 
element $\gamma$. The principal final segments are exactly the final segments that contain
smallest elements. Analogously, we define an initial segment to be \bfind{principal} if it
contains a largest element.

For every subset $M$ of $\Gamma$, we will denote by $M^c$ its complement in $\Gamma$, i.e., 
\[
M^c=\Gamma\setminus M\>.
\] 
We note that if $S\subsetneq\Gamma$ is a final segment of $\Gamma$, then $S^c=\{\gamma\in
\Gamma\mid \gamma<S\}$ is an initial segment and $(S^c,S)$ is a cut.

\pars
The proof of the following lemma is straightforward.
\begin{lemma}                                          \label{infsup}
Take a final segment $S$ of the ordered set $\Gamma$. 
\sn
1) If $S$ is nonprincipal but has infimum 
$\gamma$ in $\Gamma$, then $S^c$ is principal with largest element $\gamma$. 
\sn
2) We have that $S^c$ has a supremum if and only if $S$ has an infimum in $\Gamma$.
\sn
3) If $\,\Gamma$ is densely ordered, then an element $\gamma\in \Gamma$ is an infimum of
$S$ if and only if it is a supremum of $S^c$. \qed
\end{lemma}

Every nonempty ordered set $(\Gamma,\leq)$ without smallest or largest element carries a
topology with the open intervals $(\alpha,\beta)$, $\alpha,\beta\in\Gamma$ with
$\alpha<\beta$, as its basic open sets. We leave the proof of the following lemma
to the reader.
\begin{lemma}                          \label{cl-open}
Take a final segment $S$ in $\Gamma$.
\sn
1) If $S$ is principal, then it is closed, and it is clopen if and only if $\,\Gamma$ is
discretely ordered.
\sn
2) If $S$ does not have an infimum in $\Gamma$, then it is clopen.
\sn
3) If $S$ is nonprincipal but has an infimum in $\Gamma$, then $S$ is open, but not 
closed. 
\sn
4) The final segment $S$ is closed if and only if $S$ is principal or does not have an 
infimum in $\Gamma$.     \qed
\end{lemma}

Given a final segment $S$ in $\Gamma$, we set 
\[
\widehat S\>:=\>\left\{
\begin{array}{ll}
S\cup\{\gamma\} & \mbox{if $S$ has infimum $\gamma$ in $\Gamma$, and} \\
S & \mbox{otherwise.}
\end{array}\right.
\]
Then $\widehat S$ is the closure of $S$ in $\Gamma$. We will also write 
$S\,\widehat{\ }\,$ for $\widehat S$ in case the expression $S$ is too long to accommodate a hat.

\parm
The function
\begin{equation}                  \label{Semb}
\Gamma\ni \gamma\>\mapsto\> \gamma^-\in \Gamma\fs 
\end{equation}
is an embedding of $(\Gamma,\leq)$ in $(\Gamma\fs,\subseteq)$. Also sending $\gamma$ 
to $\gamma^+$ produces such an embedding, but throughout this paper we will work with the
embedding (\ref{Semb}).

\mn
%
%
%
\subsection{Final segments in ordered abelian groups}          \label{sectfsoag}
\mbox{ }\sn
From now on, we will work with ordered abelian groups $(G,\leq)$ and their final segments. 
The definitions and results of the previous section remain valid for $\Gamma=G$. We note
that $(G,\leq)$ is discretely ordered if and only if it contains a smallest positive
element, and it is densely ordered otherwise.

If $M$, $M_1$ and $M_2$ are nonempty subsets of $G$, then we set $M_1+M_2=\{\alpha+
\beta\mid \alpha\in M_1\,,\,\beta\in M_2\}$ and $M_1-M_2=\{\alpha-\beta\mid \alpha\in 
M_1\,,\,\beta\in M_2\}$. From this, the sets $\alpha+M$, $\alpha-M$ and $M-\alpha$ are 
obtained via replacing $\alpha$ by $\{\alpha\}$. We set $-M:=0-M$.

If $S_1$ and $S_2$ are final segments, then so is $S_1+S_2\,$. With this addition, 
$G\uparrow$ is a commutative monoid with neutral element $0^-$. It is in general not a group 
and can even contain itempotents other than $0^-$; see part 1) of Lemma~\ref{densdiscr} and 
part 3) of Lemma~\ref{+0^+}.
For example, the lexicographically 
ordered group $G=\Q\times\Z$ has the proper nontrivial subgroup $H=\{0\}\times\Z$, which
gives rise to two idempotent final segments. Indeed, since $H+H=H$, both
\[
\{\alpha\in G\mid \forall \beta\in H: \alpha> \beta\})\quad\mbox{ and }\quad\{\alpha\in G
\mid \exists \beta\in H: \alpha\geq \beta\})
\]
are idempotent, and it can be shown that they are the only
idempotent final segments in $G$.

\parm
For $\Gamma=G$ the embedding (\ref{Semb}) is a homomorphism, that is,
\begin{equation}               \label{sumpr}
(\gamma_1+\gamma_2)^-\>=\> \gamma_1^-+\gamma_2^-\>. 
\end{equation}
We note that 
\begin{equation}                      \label{s^-+0^+}
\gamma^-\,+\,0^+\>=\>\gamma^+\>. 
\end{equation}
Further, the reader may prove:
\begin{lemma}                              \label{densdiscr}
1) We have that $G$ is densely ordered if and only if $0^+ +0^+=0^+$.
\sn
2) Assume that $G$ is discretely ordered. 
If the final segment $S$ of $G$ has an infimum, then this is the smallest element of $S$
and $S$ is principal.             \qed
\end{lemma}

\pars
%
\begin{proposition}                                  \label{sumfs}
Take final segments $S_1$ and $S_2$ of $G$.
\sn
1) The sum of two final segments is again a final segment. 
The sum of two principal final segments is again a principal final segment. The sum of two 
final segments of which at least one is nonprincipal is a nonprincipal final segment. 
\sn
2) If $S_2$ is principal but $S_1$ is not, then equation (\ref{eqT}) has no solution.
\sn
3) If $S_1$ has infimum $\alpha_1$ and $S_2$ has infimum $\alpha_2\,$, then $S_1+S_2$ has
infimum $\alpha_1+\alpha_2$.
\sn
4) If $S$ is a final segment of $G$ and $\alpha\in G$, then
\begin{equation}            \label{a+a^-+}
\alpha+S\>=\>\alpha^-+S\>.
\end{equation}
\end{proposition}
\begin{proof}
The proof of parts 1) and 4) are straightforward. Part 2) follows immediately from part 1).
\sn
3): If both $S_1$ and $S_2$ are principal, then the assertion follows from (\ref{sumpr}). 
If at least one of them is nonprincipal, then by part 2) of Lemma~\ref{densdiscr}, $G$ must 
be densely ordered. Now we have to consider the following two cases (up to symmetry). If
$S_1=\alpha_1^-$ and $S_2=\alpha_2^+$, then by (\ref{s^-+0^+}) and (\ref{sumpr}),
\[
S_1\,+\,S_2\>=\> \alpha_1^- \,+\, \alpha_2^+\>=\> \alpha_1^- \,+\, \alpha_2^- \,+\, 0^+
\>=\> (\alpha_1+\alpha_2)^- \,+\, 0^+ \>=\> (\alpha_1+\alpha_2)^+\>.
\]
If $S_1=\alpha_1^+$ and $S_2=\alpha_2^+$, then by (\ref{s^-+0^+}) and (\ref{sumpr})
together with part 1) of Lemma~\ref{densdiscr},
\[
S_1\,+\,S_2\>=\> \alpha_1^+ \,+\, \alpha_2^+\>=\> \alpha_1^-\,+\, 0^+ \,+\, \alpha_2^-
\,+\, 0^+ \>=\> (\alpha_1+\alpha_2)^- \,+\, 0^+ \>=\> (\alpha_1+\alpha_2)^+\>.
\]
\end{proof}

Via the embedding (\ref{Semb}) we can view $(G,\leq)$ as an ordered subgroup of the ordered
monoid $(G\fs, \subseteq)$, consisting of all principal final segments in $G\fs$. 
From this, we obtain the following positive answers to questions (QFS1) and (QFS2) in the 
case of principal final segments. The proofs are straightforward.
\begin{proposition}                            \label{S_1pr}
Take elements $\alpha$, $\alpha_1^-$ and $\alpha_2^-$ in $G$.
\sn
1) The function
\begin{equation}                      \label{+alpha}
G\fs\ni T\>\mapsto\> \alpha^- +T \in G\fs
\end{equation}
is a bijection with inverse $G\fs\ni T\mapsto (-\alpha)^- +T\in G\fs$.
\sn
2) If $S$ is any final segment of $G$, then the unique final segment $T$ of $G$ such that
$\alpha^-+T=S$ is $T=S+(-\alpha)^-=S-\alpha$.
%
Consequently, if $S_1$ and $S_2$ are final segments of $G$ and $S_1$ is principal, say
$S_1=\alpha^-$, then equation (\ref{eqT}) has the unique solution
\begin{equation}
T\>=\> S_2\,\ms\, S_1\>=\>S_2-\alpha\>.
\end{equation}
\end{proposition} \qed

Part 1) of this proposition shows that the function
\begin{equation}                  \label{T+}
G\fs\ni T\>\mapsto\>S+T\in G\fs
\end{equation}
is injective if $S\in G\fs$ is principal. Part 3) of the following Lemma shows that if
$S$ is nonprincipal, then the function (\ref{T+}) is never injective, provided that $G$ 
is nontrivial.
\begin{lemma}                              \label{+0^+}
Take final segments $S$ and $T$ of $G$.
\sn
1) We have that $0^+ +T=\{\beta\in T\mid \alpha<\beta \mbox{ for some } \alpha\in T\}$.
\sn
2) If there is $\alpha\in G$ such that $T=\alpha^-$, then $0^+ +T=\alpha^+$.
\sn
3) If $T$ is nonprincipal, then $0^+ +T=T$.
\sn
4) If $S$ is nonprincipal, then $S+T=S+\widehat T$.
\end{lemma}
\begin{proof}
The proof of statements 1) and 2) is straightforward. If $T$ is nonprincipal, then for
every $\beta\in T$ there is some $\alpha\in T$ such that $\alpha<\beta$.
Therefore, part 3) follows from part 1).
\sn
4): Since $S$ is nonprincipal, part 3) yields that $S=S+0^+$.
If $T$ is closed in $G$, then $\widehat T=T$ and there is nothing to show.
If $T$ is not closed in $G$, then by part 4) of Lemma~\ref{cl-open}, $T$
has an infimum $\alpha\in G\setminus T$. Thus, $\widehat T=\alpha^-$ and $0^+ +\widehat
T=\alpha^+=T$, so that $S+\widehat T=S+0^+ +\widehat T=S+T$.
\end{proof}

\begin{corollary}
For every $\alpha\in G$ and nonprincipal final segment $T$ of $G$,
\begin{equation}                     \label{+alpha^-^+}
\alpha+T\>=\>\alpha^- +T\>=\>\alpha^+ +T\>.
\end{equation}
Hence the function
\begin{equation}                      \label{+alpha^+}
T\>\mapsto\> \alpha^+ +T
\end{equation}
is a bijection on the set of all nonprincipal final segments of $G$, and there it is equal
to the function (\ref{+alpha}).
\end{corollary}
\begin{proof}
The first equality of (\ref{+alpha^-^+}) follows from equation (\ref{a+a^-+}). The second
equality holds since $\alpha^- +T=\alpha^- +(0^+ +T)=(\alpha^- +0^+) +T)= \alpha^+ +T$ by
part 3) of Lemma~\ref{+0^+} and equation (\ref{s^-+0^+}). The final statement follows
from the second equality of (\ref{+alpha^-^+}) together with part 1) of
Proposition~\ref{S_1pr} and part 1) of Proposition~\ref{sumfs}.
\end{proof}

\pars
For the conclusion of this section, we note the following fact; its proof is 
straightforward. If $M\subseteq G$ and $\alpha\in G$, then 
\begin{equation}                          \label{0^-+M}
\alpha^-\,+\,M\>=\>(\alpha+M)^- \>.
\end{equation}

\mn
%
%
%
\subsection{Subtraction of final segments}       
\mbox{ }\sn
The notion of subtraction of Dedekind cuts has been defined in articles in the
literature, such as \cite{FM} and \cite{KKF}. When upper cut set addition is used, then 
for two Dedekind cuts 
$\Lambda_1=(\Lambda_1^L,\Lambda_1^R)$ and $\Lambda_2=(\Lambda_2^L,\Lambda_2^R)$, their
difference $\Lambda_2-\Lambda_1$ has been defined to be the cut with upper cut set
$\Lambda_2^R-\Lambda_1^L$, which is the sum of the final
segments $\Lambda_2^R$ and $-\Lambda_1^L$ and thus again a final segment.

Analogously, we can define a difference of two final segments $S_2$ and $S_1$
of $G$ as
\begin{equation}                    \label{classdiff}
S_2 \,-\, S_1^c\>=\> S_2\,+\, (-S_1^c)\>=\> S_2\,+\, \{\gamma\in G\mid\gamma>-S_1\}\>.
\end{equation}
Note that $(-S)^c=-(S^c)$ since $\gamma\in (-S)^c\Leftrightarrow \gamma>-S \Leftrightarrow
-\gamma<S \Leftrightarrow -\gamma\in S^c\Leftrightarrow \gamma\in -(S^c)$. Therefore we
will omit the brackets and just write ``$\,-S^c\,$''.

Definition {\ref{classdiff}) has for instance been used in \cite{FM,KKF}. However, it does
not extend the difference operation of the embedded copy of $G$. Indeed, for every 
$\alpha\in G$ we have
\[
\alpha^-\,-\,(\alpha^-)^c \>=\>\alpha^-\,+\,\{\gamma\in G\mid \gamma>-\alpha^-\}\>=\>
\alpha^-\,+\,\{\gamma\in G\mid \gamma>-\alpha\}\>=\> 0^+\>.
\]
In order to obtain a difference operation that extends the one of $G$, replacing
``$>$'' by ``$\geq$'' in (\ref{classdiff}) we define the
\bfind{(modified) final segment subtraction} of two final segments $S_2$ and $S_1$ as
\[
S_2 \,\mfs\, S_1\>:=\> S_2\,+\,\{\gamma\in G\mid \gamma\geq -S_1\}\>.
\]
In other words, we are replacing the complement $S_1^c$ in (\ref{classdiff}) by the
``modified complement'' $\{\gamma'\in G\mid \gamma'\leq S_1\}$.
This subtraction extends the subtraction of $G$: for arbitrary $\alpha,\beta\in G$,
\begin{equation}                         \label{diffG}
\alpha^-\,\mfs\,\beta^- \,=\, \alpha^-\,+\,\{\gamma\mid \gamma\geq -\beta^-\}
\,=\, \alpha^-\,+\,\{\gamma\mid \gamma\geq -\beta\}\,=\,(\alpha-\beta)^-\,.
\end{equation}

To simplify notation, for a final segment $S$ of $G$ we set
\[
\Delt S\>:=\> \{\gamma\in G\mid \gamma\geq -S\}\>=\> \widehat{-S^c}\>.
\]

%
\begin{lemma}                               \label{Delt S}
1) For each $\alpha\in G$,
\begin{equation}                \label{Deltpr}
\Delt \alpha^-\>=\>(-\alpha)^-\>.
\end{equation}
2) If $S$ is a nonprincipal final segment of $G$, then $\Delt S=-S^c$.
\sn
3) For every final segment $S$ of $G$, we have $\Delt (\Delt S)=\widehat S$.
\end{lemma}
\begin{proof}
1): We have $\Delt \alpha^-=\{\gamma\in G\mid \gamma\geq -\alpha^-\}=\{\gamma\in G\mid
\gamma\geq-\alpha\}=(-\alpha)^-$.
\sn
2): If $S$ is a nonprincipal final segment of $G$, then $\gamma\geq
-S\Leftrightarrow -\gamma\leq S\Leftrightarrow -\gamma<S\Leftrightarrow -\gamma\in S^c
\Leftrightarrow \gamma\in -S^c$, so $\Delt S=-S^c$.
\sn
3): If $S=\alpha^-$ for some $\alpha\in G$, then $\Delt (\Delt S)=\Delt((-\alpha)^-)=
\alpha^-=S=\widehat{S}$, where we have used part 1) twice.
If $S$ is a nonprincipal final segment of $G$, then $\Delt (\Delt S)=\Delt(-S^c)=
(-(-S^c)^c)\,\widehat{\ }\,=(-(-S))\,\widehat{\ }\,=\widehat{S}$, where we have used part
2) and the definition of $\Delt$.
\end{proof}

\begin{remark}
{\rm The final segment $\alpha^+$ is nonprincipal if and only if $G$ is densely ordered; 
if $G$ has smallest positive element $\gamma$, then $\alpha^+=(\alpha+\gamma)^-$. 
Therefore, we have
\[
\Delt \alpha^+\>=\>\left\{\begin{array}{ll}
(-\alpha)^- & \mbox{if $G$ is densely ordered,}\\    
(-\alpha-\gamma)^- & \mbox{if $G$ is discretely ordered.}
\end{array}\right. 
\]
}
\end{remark}

\begin{lemma}                             \label{+-}
Take final segments $S_1\,$ and $S_2$ of $G$.
\sn
1) We have that
\[
S_2 \,\mfs\, S_1
\>=\> S_2\,+\, \Delt S_1\>.
\]
In particular, $0^- \,\mfs\, S_1=\Delt S_1\,$.
\sn
2) The sets $\Delt S_1$, $S_2 \mfs S_1\,$, $-S_1^c$ and $S_2-S_1^c$ are again final
segments of $G$.
\sn
3) If $S_1$ is nonprincipal, then
\begin{equation}                    \label{+-eq2}
S_2 \,\mfs\, S_1\>=\>S_2\,-\,S_1^c\>.
\end{equation}
4) If $S_1$ is principal, say $S_1=\alpha^-$ for some $\alpha\in G$, then
\[
S_2 \,\mfs\, S_1\>=\>S_2\,+\,(-\alpha)^-\>=\>S_2-\alpha\>.
\]
\sn
5) The following holds:
\begin{equation}                                \label{eq10}
S_1\,+\,(S_2\,\mfs\, S_1)\>=\> (S_1 \,\mfs\, S_1)\,+\,S_2\>.
\end{equation}
\sn
6) The following holds:
\[
S_2 \,\mfs\, (0^-\,\mfs\, S_1)\>=\> S_2\,+\, \widehat{S_1}\>.
\]
In particular,
\begin{equation}                   \label{0^-0^-S}
0^- \,\mfs\, (0^-\,\mfs\, S_1)\>=\> \widehat{S_1}\>.
\end{equation}
\end{lemma}
\begin{proof}
1): This follows from the definitions of $S_2\mfs S_1$ and $\Delt S_1\,$.
\sn
2): For $\Delt S_1$ this holds by definition, so for $S_2 \mfs S_1$ it follows from part
1). For $-S_1^c$ we observe that $S_1^c$ is an initial segment, hence $-S_1^c$ is a final
segment and the same holds for $S_2-S_1^c=S_2+(-S_1^c)$.
\sn
3): This follows from part~1) together with part 2) of Lemma~\ref{Delt S}.
\sn
4): This follows from part~1) together with equation (\ref{Deltpr}) and part 4) of Proposition~\ref{sumfs}.
\sn
5): The proof is straightforward, using part 1).
\sn
6): Using parts 1) and part 3) of Lemma~\ref{Delt S}, we compute:
\[
S_2 \,\mfs\, (0^-\,\mfs\, S_1)\>=\>S_2 \,\mfs\, \Delt S_1 \>=\> S_2 \,+\, \Delt
(\Delt S_1) \>=\> S_2 \,+\, \widehat S_1\>.
\]
\end{proof}

\begin{remark}
\rm The final segment $S_2 \mfs S_1$ is closed, except in the following cases. If 
$S_1=\alpha^-$ and $S_2=\beta^+$ for some $\alpha,\beta\in G$ and $\beta$ has no
immediate successor in $G$, then $\beta-\alpha$ has no immediate successor in $G$ and
therefore, $S_2 \mfs S_1=(\beta-\alpha)^+$ is not closed. If $S_2$
and $S_1$ have no infimum in $G$, then $S_2 \mfs S_1$ may or may not have an infimum 
in $G$. For instance, if $G=\Q$ and $S_1=S_2=\{\alpha\in\Q\mid \alpha>\pi\}$, then 
$\Delt S_1=\{\alpha\in\Q\mid \alpha>-\pi\}$ and $S_2 \mfs S_1=0^+$ which is not closed
although $S_1$ and $S_2$ are.
\end{remark}

\pars
The question arises whether the difference operation $\mfs$ provides some positive answer 
to our questions (QFS1) and (QFS2). As part 2) of Proposition~\ref{S_1pr} in conjunction
with part 4) of Lemma~\ref{+-} shows, it indeed gives a complete answer to these questions
in the case where $S_1$ is principal. It now remains to deal with the case where $S_1$ is
nonprincipal. This case is much more complex and we need some preparation.

\mn
%
%
%
\subsection{Invariance groups}          \label{sectig}
\mbox{ }\sn
Throughout, we let $G$ be an arbitrary ordered abelian group. For a subset $M$ of $G$, we 
define its \bfind{invariance group} to be
\[
\cG(M)\>:=\>\{\gamma\in G\mid M+\gamma=M\}\>. 
\]

\begin{lemma}                               \label{lemGS}
Take any nonempty subset $M\subseteq G$.
\sn
1) \ We have that $\cG(M)$ is a subgroup of $G$.
\sn
2) \ We have that
\begin{equation}                            \label{GS=+-}
\cG(M)\>=\>\{\gamma\in G\mid M+\gamma\subseteq M \mbox{\rm\ and } M-\gamma\subseteq M\}\>.
\end{equation}
3) \ The following holds:
\[
\cG(M)\>=\>\cG(-M)\>=\>\cG(M^c)\>. 
\]
4) \ We have $0\in M$ if and only if $\cG(M)\subseteq M$.
\sn
5) \ If $M$ is convex, then $\cG(M)$ is a convex subgroup of $G$.
\sn
6) \ If $M$ has an infimum in $G$, then $\cG(M)=\{0\}$.
In particular, for every $\alpha\in G$,
\begin{equation}                         \label{G(+-)}
\cG(\alpha^-) \>=\>\cG(\alpha^+)\>=\>\cG(\{\alpha\})\>=\>\{0\}\>.
\end{equation}
\end{lemma}
\begin{proof}
1): \ Take $\alpha,\beta\in \cG(M)$. Then $M+\alpha+\beta=M+\beta=M$, whence $\alpha+\beta
\in \cG(M)$. Further, $M-\alpha=M+\alpha-\alpha=M$, whence $-\alpha\in \cG(M)$. This proves
that $\cG(M)$ is a group. 
\sn
2): \ The inclusion ``$\subseteq$'' in (\ref{GS=+-}) follows from the definition of the
invariance group and part 1) of our lemma. If $M+\gamma\subseteq M$ and $M-\gamma\subseteq 
M$, then $M+\gamma\subseteq M=M-\gamma+\gamma\subseteq M+\gamma$, whence $M+\gamma=M$. 
This proves the inclusion ``$\supseteq$'' in (\ref{GS=+-}).
\sn
3): \ The first equality holds since $\gamma\in \cG(M)\Leftrightarrow -\gamma\in \cG(M)
\Leftrightarrow M-\gamma=M\Leftrightarrow -M+\gamma=-M\Leftrightarrow \gamma\in \cG(-M)$, 
where we used that $\cG(M)$ is a group. 

To prove the second equality, we observe that if $\gamma\notin\cG(M)$, then there exists 
some $\alpha\in M$ such that $\alpha+\gamma\notin M$, whence $\alpha+\gamma\in M^c$. Then 
$\alpha+\gamma-\gamma=\alpha\notin M^c$, so $\gamma\notin
\cG(M^c)$ by (\ref{GS=+-}). This shows that $\cG(M^c)\subseteq\cG(M)$. The reverse inclusion
follows by substituting $M^c$ for $M$ in our argument.
\sn
4): \ If $0\in M$, then $\cG(M)=0+\cG(M)\subseteq M+\cG(M)=M$. The converse is trivial.
\sn
5): \ By part 1), it suffices to show that $\cG(M)$ is convex. Take $\alpha,\beta\in 
\cG(M)$ and $\gamma\in G$ with $\alpha\leq \gamma \leq \beta$. Then for all $\delta\in 
M$, $\delta+\alpha\leq \delta+\gamma\leq \delta+\beta$ and $\delta+\alpha,\delta+\beta\in 
M$. Hence $\delta+\gamma\in M$ by convexity of $M$. This proves that $M+\gamma\subseteq M$. 
In the same way, we also obtain that $M-\gamma\subseteq M$ because $-\alpha,-\beta\in 
\cG(M)$ (since $\cG(M)$ is a group) and $-\beta\leq -\gamma \leq -\alpha$. By part 2) it
follows that $\gamma\in \cG(M)$.
\sn
6): \ Assume that $\alpha\in G$ is the infimum of $M$ in $G$. Then for every positive 
$\gamma\in G$ there is $\beta\in M$ such that $0\leq \beta-\alpha<\gamma$. Hence $\beta
-\gamma<\alpha$ so that $\beta-\gamma\notin M$ and thus
$\gamma\notin \cG(M)$ by (\ref{GS=+-}). This proves that $\cG(M)=\{0\}$. 
\end{proof}

Observe that in the linear ordering given by inclusion on the set of all convex subgroups,
\[
H_1\cup H_2\>=\>H_1+H_2\>=\>\max\{H_1,H_2\}
\]
for every two convex subgroups $H_1$ and $H_2$ of $G$.
\begin{lemma}                               \label{GS+S}
1) \ If $M_1$, $M_2$ are arbitrary subsets of $G$, then
\[
\cG(M_1)\cup \cG(M_2)\subseteq \cG(M_1+M_2)\>.
\]
\sn
2) \ If both $S_1$, $S_2$ are initial segments or final segments of $G$, then
\[
\cG(S_1+S_2)\>=\>\cG(S_1)+\cG(S_2)\>=\> \cG(S_1)\cup \cG(S_2)\>.
\]
3) \ If $\alpha\in G$ and $S$ is a final segment of $G$, then
\[
\cG(\alpha+S)\>=\>\cG(S)\>.
\]
\sn
4) \ If $S$ is a final segment and $H$ a convex subgroup of $G$, then $\cG(H)=H$ and
\[
\cG(S+H)\>=\>\cG(S)\cup H\>. 
\]
\sn
5) \ For each $\alpha\in G$ and every convex subgroup $H$ of $G$,
\[
\cG(\alpha+H)\>=\>\cG((\alpha+H)^-)\>=\>\cG((\alpha+H)^+)\>=\>H\>.
\]
\sn
6) \ For each $\alpha\in G$ and every final segment $S$ of $G$, we have 
\[
\alpha+S\subseteq S\>\Leftrightarrow\> \alpha\in\cG(S)^- \;\mbox{ and }\; \alpha+S\subsetneq 
S\Leftrightarrow \alpha\in\cG(S)^+\>.
\]
\end{lemma}
\begin{proof}
1): \ Take $\gamma_1\in \cG(M_1)$ and $\gamma_2\in \cG(M_2)$. Then for all $\alpha_1\in 
M_1$ and $\alpha_2\in M_2\,$, $(\alpha_1+\alpha_2)+\gamma_1=(\alpha_1+\gamma_1)+\alpha_2\in
M_1+M_2$ and $(\alpha_1+\alpha_2)+\gamma_2=\alpha_1+(\alpha_2+\gamma_2)\in M_1+M_2$. Hence,
$\gamma_1,\gamma_2\in \cG(M_1+M_2)$.
\sn
2): \ In view of part 1), it suffices to show that $\cG(S_1+S_2)
\subseteq \cG(S_1)\cup \cG(S_2)$. Let $S_1$ and $S_2$ be final
segments of $G$; the proof for initial segments is similar. Take a
positive $\gamma\in G$ such that $\gamma\notin \cG(S_1)\cup \cG(S_2)$. Then there are 
$\alpha_1\in S_1$ and $\alpha_2\in S_2$ such that $\alpha_1-\gamma\notin S_1$
and $\alpha_2-\gamma\notin S_2\,$. Since $S_1$ and $S_2$ are final segments of $G$, this
means that $\alpha_1-\gamma<S_1$ and $\alpha_2-\gamma<S_2\,$. Thus $\alpha_1+\alpha_2
-2\gamma<S_1+S_2\,$, which means that $2\gamma\notin \cG(S_1+S_2)$ and hence $\gamma\notin 
\cG(S_1+S_2)$.
\sn
3): \ 
We compute: $\cG(\alpha+S)=
\cG(\alpha^-+S)=\cG(\alpha^-)\cup\cG(S)=\{0\}\cup\cG(S)=\cG(S)$, where we have used part 
2) of our lemma together with equations (\ref{a+a^-+}) and (\ref{G(+-)})
\sn
4): \ If $H'$ is any convex subgroup of $G$, then it is comparable to $H$ by inclusion. 
Hence $H+H'=H$ holds if and only if $H'\subseteq H$. This shows that $\cG(H)=H$. Thus 
we obtain from part 1) that $\cG(S)\cup H=\cG(S)\cup\cG(H)\subseteq \cG(S+H)$. Take a
positive $\gamma\in G$ such that $\gamma\notin \cG(S)\cup H$. Then there are 
$\alpha_1\in S$ and $\alpha_2\in H$
such that $\alpha_1-\gamma< S$ and $\alpha_2-\gamma<H$. Thus $\alpha_1+\alpha_2-2\gamma
<S+H$, which means that $2\gamma\notin \cG(S+H)$ and hence $\gamma\notin \cG(S+H)$.
\sn
5): \ We have $\cG(\alpha+H)=\{\gamma\in G\mid \gamma+\alpha+H=\alpha+H\}=\{\gamma\in G\mid
\gamma+H=H\}=H$. By equation (\ref{0^-+M}), $(\alpha+H)^-=\alpha^-+H$. Hence by part 4) and
equation (\ref{G(+-)}) of Lemma~\ref{lemGS}, $\cG((\alpha+H)^-)=\cG(\alpha^-+H)=
\cG(\alpha^-)+H=H$. We observe 
that the complement of $-(\alpha+H)^+$ in $G$ is $(-\alpha+H)^-$. Applying part 3) of 
Lemma~\ref{lemGS} twice, we obtain that $\cG((\alpha+H)^+)=\cG(-(\alpha+H)^+)=
\cG((-\alpha+H)^-)=H$.
\sn
6): \ Since $\cG(S)$ is a convex subgroup of $G$, $\cG(S)^-$ is the disjoint union of
$\cG(S)$ and $\cG(S)^+$. If $\alpha
\in \cG(S)$, then $\alpha+S=S$. If $\alpha\in \cG(S)^+$, then $\alpha>0$ and therefore, 
$\alpha+S\subseteq S$; but as $\alpha\notin \cG(S)$, $\alpha+S\ne S$. This proves the
implications ``$\Leftarrow$'' in both statements.

If $\alpha\notin\cG(S)^-$, then $\alpha+S\ne S$ since $\alpha\notin\cG(S)$, hence 
$S\subsetneq \alpha+S$ since $\alpha<0$. If $\alpha\notin\cG(S)^+$, then $\alpha\in\cG(S)$ 
and $\alpha+S=S$, or $\alpha<0$ and $S\subseteq \alpha+S$. This proves the implications 
``$\Rightarrow$'' in both statements.
\end{proof}

\pars
For a final segment $S$, we define 
\[
nS\>:=\>\{\alpha_1+\ldots+\alpha_n\mid \alpha_i\in S\}\>.
\]
Note that the set $\{n\alpha\mid \alpha\in S\}$ is a final segment if and only if $G$ is
$n$-divisible; in general, $nS$ is the smallest final segment that contains 
$\{n\alpha\mid \alpha\in S\}$. By a repeated application of part 2) of the last lemma, we 
obtain:
\begin{corollary}                  \label{cGnS}
Take a final segment $S$ of $G$. Then
\[
\cG(nS)\>=\>\cG(S)\>. 
\]
\end{corollary}

From part 2) of Lemma~\ref{GS+S} and the fact that the convex subgroups of $G$ are
linearly ordered by inclusion, we obtain:
\begin{proposition}                    \label{nosol}
Take final segments $S_1\,$, $S_2$ and $T$ such that $S_1+T=S_2$. Then
$\cG(S_1)\subseteq\cG(S_1)\cup\cG(T)=\cG(S_2)$, and if $\cG(S_1)\subsetneq\cG(S_2)$,
then $\cG(T)=\cG(S_2)$.
\end{proposition}

\begin{lemma}                      \label{cGcomp}
Take a final segment $S$ of $G$.
\sn
1) The following holds:
\[
\cG(-S)\>=\>\cG(S)\>=\>\cG(S^c)\>=\>\cG(\Delt S)\>. 
\]
2) If $\cG(S)\ne\{0\}$ then $S$ is nonprincipal.
\end{lemma}
\begin{proof}
1): The first and second equality are special cases of part 3) of Lemma~\ref{lemGS}. The 
equality $\cG(S)=\cG(\Delt S)$ is seen as follows. If $S$ is nonprincipal, then by
part 2) of Lemma~\ref{Delt S}, $\Delt S=-S^c$ and the equality follows from what we have
already proved. If $S$ is principal, say $S=\alpha^-$ for some $\alpha\in G$, then
$\Delt S=(-\alpha)^-$ by (\ref{Deltpr}), whence $\cG(\Delt S)=\{0\}=\cG(S)$ by
(\ref{G(+-)}).
\sn
2): If $\cG(S)\ne\{0\}$, then there is a negative element $\alpha\in \cG(S)$ so that for
every $\beta\in S\,$, $\beta>\alpha+\beta\in S\,$; hence $S$ has no smallest element.
\end{proof}

\parm
This lemma enables us to also define the invariance groups of cuts and quasi-cuts.
\begin{proposition}
If $(\Lambda^L,\Lambda^R)$ is a quasi-cut in $G$, then
\[
\cG(\Lambda^L)\>=\> \cG(\Lambda^R)\>.
\]
If $(\Lambda^L,\Lambda^R)$ is not a cut, then these groups are trivial.
\end{proposition}
\begin{proof}
The first statement follows from the fact that for every $\gamma\in G$,
\[
\Lambda^L+\gamma\,=\,\Lambda^L\;\Leftrightarrow\;\Lambda^R+\gamma\,=\,\Lambda^R\>.
\]
If $(\Lambda^L,\Lambda^R)$ is not a cut, then $\Lambda^L\cap\Lambda^R=\{\gamma\}$ for some
$\gamma\in G$. It follows that $\Lambda^R=\gamma^-$, hence by part 2) of
Lemma~\ref{cGcomp}, $\cG(\Lambda^L)=\cG(\Lambda^R)=\{0\}$.
\end{proof}
For every quasi-cut $\Lambda=(\Lambda^L,\Lambda^R)$ we can now define
\[
\cG(\Lambda)\>:=\> \cG(\Lambda^L)\>=\> \cG(\Lambda^R)\>.
\]

\parm
In order to give answers to questions (QFS1), (QFS2) and (QFS3) in the case where $\cG(S_1)
=\cG(S_2)=\{0\}$ in Section~\ref{sectpa} below, we need some more preparation.
We set $G^{>0}:=\{\gamma\in G\mid \gamma>0\}=0^+$.
\begin{lemma}                           \label{lemG-cG}
Take a final segment $S$ of $G$. Then $\cG(S)$ is the largest of all 
subgroups $H$ of $G$ satisfying 
$H\cap (S-S^c)= \emptyset$. Further,
\begin{equation}                            \label{G-cG}
S-S^c\>=\>G^{>0}\setminus \cG(S)\>=\>\cG(S)^+\>.
\end{equation}
\end{lemma}
\begin{proof}
If $\gamma=\alpha-\beta$ with $\alpha\in S$ and $\beta\in S^c$, then $\alpha-\gamma
=\beta\notin S$, so $-\gamma\notin \cG(S)$. Since $\cG(S)$ is a subgroup of $G$, this
implies that $\gamma\notin\cG(S)$. Conversely, if $H$ is a subgroup of $G$ properly 
containing $\cG(S)$, then there is $\gamma\in H^{>0}\setminus\cG(S)$. 
Then $-\gamma\notin\cG(S)$ and there is $\alpha\in S$ such that $\alpha-\gamma\notin S$.
This implies $\beta=\alpha-\gamma\in S^c$, whence $\gamma=\alpha-\beta\in S-S^c$.
\pars
Now we prove equation (\ref{G-cG}). Since $\alpha>\beta$ whenever $\alpha\in S$ and
$\beta\in S^c$, it follows that $S-S^c\subseteq G^{>0}\setminus\cG(S)$. For the converse,
assume that $\gamma\in G^{>0} \setminus \cG(S)$. Then there is some $\beta\in S^c$ such
that $\beta+\gamma\in S$ and therefore, $\gamma\in S-S^c$.
\end{proof}

\begin{corollary}                          \label{corsols}
Take final segments $S_1$ and $S_2$ of $G$ and assume that $S_1$ is nonprincipal.
\sn
1) We have
\[
S_1 \,\mfs\, S_1\>=\> S_1-S_1^c\>=\> \cG(S_1)^+\>, 
\]
so that
\begin{equation}                       \label{eqcorsols}
S_1\,+\,(S_2-S_1^c)\>=\>S_1\,+\,(S_2\,\mfs\, S_1) \>=\> \cG(S_1)^+ +S_2\>\subseteq S_2\>.
\end{equation}
\sn
2) If $\cG(S_1)=\{0\}$ and $S_2$ is nonprincipal, then
$S_1+(S_2- S_1^c)=S_2\,$.
\sn
3) If $\cG(S_1)\subsetneq \cG(S_2)$, then $S_1+(S_2-S_1^c)=S_2\,$.
\end{corollary}
\begin{proof}
1): This follows from equation (\ref{+-eq2}) in part 3) of Lemma~\ref{+-} together with
Lemma~\ref{lemG-cG}. The equalities in (\ref{eqcorsols}) follow from part 1) together with
parts 3) and 5) of Lemma~\ref{+-}.
The inclusion holds since $\cG(S_1)^+$ only contains positive
elements.
\sn
2): This follows from part 1) together with part 3) of Lemma~\ref{+0^+}
\sn
3): 
Since $\cG(S_1)\subsetneq \cG(S_2)$, there is some $\alpha\in \cG(S_1)^+\cap
\cG(S_2)$. We obtain that $S_2=\alpha+S_2\subseteq
\cG(S_1)^++S_2\,$. In view of (\ref{eqcorsols}), this proves statement 3).
\end{proof}

\parm
Finally, for application in Section~\ref{sectM(I)} we prove the following facts:
\begin{proposition}                        \label{S1S2}      
Take final segments $S$, $S_1$ and $S_2$ of $G$. 
\sn
1) We have that $\cG(S)^+$ does not contain $S$ if and only if $0\in S$.
\sn
2) The set $\{\alpha\in G\mid \alpha+\cG(S)^+\subseteq S\}$ is equal to $S$ if there is no
$\gamma\in G$ such that $\cG(S)^+=\gamma+S$; otherwise, it properly contains $S$.
\sn
3) If $\cG(S_2)\subsetneq \cG(S_1)$, then there is $\alpha\in G$ such that $S_1+S_2=
S_1+\alpha$.
\sn
4) The set of all $\alpha\in G$ such that $\alpha+S\subseteq G^{>0}$ is equal to $-S^c$.
\sn
5) The final segment $\cG(S)^+$ is the union of all $\alpha+S\subseteq G^{>0}$, 
$\alpha\in S$.
\end{proposition}
\begin{proof}
1): If $\cG(S)^+$ does not contain $S$, then $S\setminus\cG(S)^+\ne\emptyset$, so
there is some $\alpha\in S\cap\cG(S)$. Since $\cG(S)$ is a group, it follows that 
$-\alpha\in\cG(S)$. Therefore $S=-\alpha+S$, whence $0=\alpha-\alpha\in S$. The converse
holds since $0\in\cG(S)^+$ is impossible.
\sn
2): Since $\cG(S)^+$ only contains positive elements, $S$ is always contained in
$\{\alpha\in G\mid \alpha+\cG(S)^+\subseteq S\}$. Assume that there is no $\gamma\in G$
such that $\cG(S)^+=\gamma+S$. Take any $\alpha\in G$ such
that $\alpha+\cG(S)^+\subseteq S$, i.e., $\cG(S)^+\subseteq -\alpha+S$. Then by assumption, 
$\cG(S)^+\subsetneq -\alpha+S$. Since $\cG(-\alpha+S)=\cG(S)$ by part 3) of 
Lemma~\ref{GS+S}, we have $\cG(-\alpha+S)^+=\cG(S)^+$. Then $0\in -\alpha+S$ by part 1) 
of our Proposition, so $\alpha\in S$.

Now assume that $\cG(S)^+=\gamma+S$ for some $\gamma\in G$. Then $-\gamma+\cG(S)^+
\subseteq S$, but $-\gamma\notin S$ since otherwise $0\in\cG(S)^+$, which is impossible.
\sn
3): Take $\beta\in\cG(S_1)\setminus\cG(S_2)$ with $\beta>0$. Then $\beta+S_2\subsetneq
S_2$ and there is $\alpha\in S_2\setminus \beta+S_2\,$, so that $\beta+S_2\subseteq
\alpha^-\subseteq S_2\,$. Now $S_1+S_2=(S_1+\beta)+S_2=S_1+(\beta+S_2)\subseteq S_1
+\alpha^-\subseteq S_1+S_2\,$, whence $S_1+S_2=S_1+\alpha^-=S_1+\alpha$.
\sn
4): This holds since $\alpha+S\subseteq G^{>0}\Leftrightarrow \alpha+S>0\Leftrightarrow 
S>-\alpha\Leftrightarrow -\alpha\in S^c\Leftrightarrow \alpha\in -S^c$.
\sn
5): By part 4), $\bigcup\{\alpha+S\mid \alpha\in G\mbox{ with } \alpha+S\subseteq G^{>0}\}
=\bigcup\{\alpha+S\mid \alpha\in -S^c\}=-S^c+S$. By equation (\ref{G-cG}) of 
Lemma~\ref{lemG-cG}, this is equal to $\cG(S)^+$.
\end{proof}

\mn
%
%
%
\subsection{Partial answers to questions (QFS1), (QFS2) and (QFS3)}   \label{sectpa}
\mbox{ }\sn
We are now in the position to give an answer to questions (QFS1), (QFS2) and (QFS3) in the 
important special case where $\cG(S_1)=\cG(S_2)=\{0\}$. We will use this answer later
together with the techniques which we will develop in Section~\ref{sectmod} to settle the
general case where $\cG(S_1)$ can be any proper convex subgroup of $G$.
\begin{lemma}                        \label{partsol}
Take final segments $S_1$ and $S_2$ of $G$ and assume that $S_1$ is nonprincipal. If
$\cG(S_2)=\{0\}$ and $T=T_0$ is a solution of (\ref{eqT}), then $T=\widehat{T_0}$ 
is the largest solution of (\ref{eqT}). 
%
\end{lemma}
\begin{proof}
Since $S_1+T_0=S_2\,$, by part 4) of Lemma~\ref{+0^+} also $S_1+\widehat{T_0}=S_2$ 
holds. Suppose that $T'$ is a larger solution, and pick some $\alpha'\in T'\setminus
\widehat{T_0}$. If $\widehat{T_0}$ has a minimum, then denote it by $\alpha$. If
$\widehat{T_0}$ has no minimum, then $\widehat{T_0}=T_0$ has no infimum and $G\setminus 
T_0$ has no supremum, so there is $\alpha\in G$ such that $\alpha'<\alpha<T_0\,$. In both
cases, we have that $\alpha\leq \widehat{T_0}$. Consequently,
\[ 
S_2\>=\> S_1+\widehat T_0\>\subseteq\> S_1+\alpha \>\subseteq\> S_1+\alpha' \>\subseteq\>
S_1+T'\>=\> S_2\>,
\]
so equality holds everywhere. It follows that $S_2+\alpha'-\alpha=S_1+\alpha+\alpha'-\alpha
=S_1+\alpha'=S_2\,$, whence $\alpha'-\alpha\in \cG(S_2)=\{0\}$. Thus $\alpha'=\alpha$, a
contradiction.
\end{proof}

\begin{remark}                              \label{cex}
\rm If $S_1$ is principal, then $\widehat{T_0}$ will not always be a solution, even if 
$\cG(S_2)$ is trivial. As a simple counterexample, assume that $G$ is densely ordered and
consider the equation $0^-+T=0^+$. Then $T=0^+$ is a solution, but $T=0^-$ is not.
\end{remark}

\begin{proposition}                         \label{solG=0}
Take final segments $S_1$ and $S_2$ of $G$. Assume that $\cG(S_1)=\cG(S_2)=\{0\}$ and that 
$S_2$ is nonprincipal if $S_1$ is nonprincipal. Then $T=S_2\ms S_1$ is the largest solution
of (\ref{eqT}), and we have
\begin{equation}
S_2\,\ms\, S_1\>=\>\left\{
\begin{array}{ll}
S_2-\alpha & \mbox{if $S_1=\alpha^-$ for some $\alpha\in G$, } \\
(S_2-S_1^c)\,\widehat{ } & \mbox{if $S_1$ is nonprincipal.}
\end{array}\right.
\end{equation}
\end{proposition}
\begin{proof}
As stated in the introduction, $T=S_2\ms S_1$ is the largest solution of $S_1+T\subseteq
S_2$; hence if a solution of equation (\ref{eqT}) exists, then $T=S_2\ms S_1$ is the
largest solution of it. For $S_1=\alpha^-$, our assertion follows from
Proposition~\ref{S_1pr}. If $S_1$
is nonprincipal, then by assumption, $S_2$ is nonprincipal, and it follows from part 2) of
Corollary~\ref{corsols} that $S_2-S_1^c$ is a solution of (\ref{eqT}). Hence our statement
for this case follows from Lemma~\ref{partsol}.
\end{proof}

\mn
%
%
%
\subsection{Ordered abelian groups modulo convex subgroups}   \label{sectmod}
\mbox{ }\sn
{\it In what follows, $G$ will be an ordered abelian group with arbitrary convex subgroup 
$H$.} The quotient $G/H$ carries a canonical ordering such that for all $\alpha,\beta\in G$, 
\begin{equation}                           \label{presleq}
\alpha\,\leq\,\beta\>\Rightarrow\> \alpha+H\,\leq\,\beta+H\>. 
\end{equation}
%
Note that under this induced ordering on $G/H$, we have $\alpha+H<\beta+H$ if and only 
if it holds for $\alpha+H$ and $\beta+H$ as subsets of $G$ as defined at the beginning 
of Section~\ref{sectprelfs}.

For the image of an element $\gamma\in G$ under the canonical epimorphism 
\[
\varphi_H:\>G\>\rightarrow\> G/H\>, 
\]
i.e., the coset $\gamma+H$, we prefer to write $\gamma/H$ 
since we also have to deal with subsets of the form $\gamma+H$ in $G$. With this notation,
$M/H=\{\gamma/H\mid \gamma\in M\}$ for every subset $M\subseteq G$. Since the epimorphism
preserves $\leq$, we have that $M/H$ is convex if $M$ is convex, and $S/H$ is a final 
segment (or initial segment) of $G/H$ if $S$ is a final segment (or initial segment,
respectively) of $G$. In what follows, we will freely make use of the following fact, 
whose proof is straightforward:
\begin{lemma}                               \label{M+H}
For each subset $M$ of $G$ the preimage $\varphi_H^{-1}(M/H)$ of $M/H$ under $\varphi_H$
is $M+H$. If $M$ is a final segment of $G$, then so is $M+H$.
\end{lemma}

\begin{proposition}                 \label{invgr/H}
Take a subset $M\subseteq G$, a final segment $S$ of $G$, and a convex subgroup $H$ of $G$. 
\sn
1) If $H\subseteq \cG (M)$, then 
\[
\cG(M/H)\>=\> \cG(M)/H\>. 
\]
2) If $\cG(S)\subseteq H$, then $\cG(S/H)$ is trivial. 
\sn
3) If $\cG(S)$ is trivial, then also $\cG(S/H)$ is trivial.
\sn
4) If $\cG(S)\subsetneq H$, then $S/H$ is principal.
\sn
5) If $G/\cG(S)$ is discretely ordered, then $S/\cG(S)$ is principal.
\end{proposition}
\begin{proof}
1): Assume that $H\subseteq \cG(M)$. Then $M+H=M$. It follows that $M+H+\gamma+H=M+\gamma$
for every $\gamma\in G$, so
we have $M/H+\gamma/H= M/H$ if and only if $M+\gamma=M$. Hence, $\cG(M/H)=\cG (M)/H$.
\sn
2): Assume that $\cG(S)\subseteq H$.
Take a convex subgroup $H'$ of $G$ containing $H$ such that $\cG(S/H)=H'/H$. That is,
$S/H+H'/H=S/H$. It follows that $S+H+H'+H=S+H$, showing that $H'$ is a subgroup of 
$\cG(S+H)$. By part 4) of Lemma~\ref{GS+S}, $\cG(S+H)=\cG(S)\cup H=H$. We conclude that 
$H'=H$ and thus, $\cG(S/H)=H'/H$ is trivial.
\sn
3): This follows from part 2) since the trivial subgroup is contained in $H$.
\sn
4): If $\cG(S)\subsetneq H$, then we have $S+H\ne S$, so there is some $\alpha\in S+H$ such
that $\alpha<S$ (since $S$ is a final segment). As $\alpha\in S+H$, it follows that 
$\alpha/H\in S/H$. We wish to show that $\alpha/H$ is the smallest element of $S/H$. Take 
any $\beta\in G$ such that $\beta/H<\alpha/H$. This means that $\beta+H<\alpha<S$,
so $S\cap (\beta+H)=\emptyset$ and therefore, $\beta/H\notin S/H$. This proves our claim. 
\sn
5): Take $\gamma\in G$ such that $\gamma/\cG(S)$ is the smallest positive element of 
$G/\cG(S)$ and therefore, $(\gamma/\cG(S))\Z$ is a convex subgroup of $G/\cG(S)$ and we 
have $(\gamma/\cG(S))^-=(0/\cG(S))^+$. Suppose that $S/\cG(S)$ is nonprincipal. Then
$(\gamma/\cG(S))^-+S/\cG(S)=(0/\cG(S))^+ +S/\cG(S)=(0/\cG(S))^- +S/\cG(S)=S/\cG(S)$
by equation (\ref{+alpha^-^+}), which implies
that $(\gamma/\cG(S))\Z\subseteq\cG(S/\cG(S))$, contradicting part 2) of our proposition.
\end{proof}

\begin{lemma}                           \label{S=cG(S)}
Take a final segment $S$ of $G$. Then $S+\cG(S)^+\ne S$ if and only if $S/\cG(S)$ is 
principal.
\end{lemma}
\begin{proof}
We have that $\cG(S)^+/\cG(S)=(0/\cG(S))^+$. Hence it follows from (\ref{s^-+0^+}) together 
with part 3) of Lemma~\ref{+0^+} that $S/\cG(S)+\cG(S)^+/\cG(S)\ne S/\cG(S)$ if and only if 
$S/\cG(S)$ is principal. If inequality holds, then we also must have $S+\cG(S)^+\ne S$. Now
assume that equality holds. Then it must also hold for the respective preimages of
$S/\cG(S)$ and $\cG(S)^+/\cG(S)$ under $\varphi_{\cG(S)}\,$. By Lemma~\ref{M+H},
$\varphi_{\cG(S)}^{-1}(S/\cG(S))=S+\cG(S)=S$. By part 5) of Lemma~\ref{GS+S} with
$\alpha=0$, $\cG(\cG(S)^+)=\cG(S)$, which shows that $\varphi_{\cG(S)}^{-1}(\cG(S)^+
/\cG(S))=\cG(S)^+ +\cG(S)=\cG(S)^+$. Therefore, the equality
$S/\cG(S)+\cG(S)^+/\cG(S)= S/\cG(S)$ implies the equality $S+\cG(S)^+ =S$.
\end{proof}

\parm
We note that if $(\Lambda^L,\Lambda^R)$ is a cut in $G$, then $\Lambda^L/H\leq\Lambda^R/H$,
hence 
\begin{equation}                            \label{L/H}
\Lambda/H\>:=\>(\Lambda^L/H,\Lambda^R/H)
\end{equation}
is a quasicut in $G/H$.

\begin{proposition}                               
The induced quasi-cut (\ref{L/H}) in $G/H$ is a cut if and only if $H\subseteq 
\cG(\Lambda)$. The invariance group of the cut $\Lambda/\cG(\Lambda)$ in $G/\cG(\Lambda)$ 
is trivial. 
\end{proposition}
\begin{proof}
The quasi-cut $\Lambda/H$ is a cut in $G/H$ if and only if $\Lambda^L/H
\cap\Lambda^R/H=\emptyset$. This holds if and only if $\Lambda^R +H
\subseteq \Lambda^R$. By definition of $\cG(\Lambda)$, this in turn
holds if and only if $H\subseteq \cG(\Lambda)$.

The last statement follows from the definition of $\cG(\Lambda)$ together with part 2) of
Proposition~\ref{invgr/H}.
\end{proof}

\mn
%
%
%
\subsection{The equation $S_1+T=S_2$ in ordered abelian groups of arbitrary rank}     
\label{sectfseq}
%
%
In this section we will first assume that $S_1\,$, $S_2$ are final segments of the
ordered abelian group $G$ and that the equation (\ref{eqT}) 
has a solution $T=T_0\,$. We determine whether it is unique, and if not, compute the
largest solution. We will use the following crucial fact, whose proof is straightforward:
\begin{lemma}                    \label{presol}
Take final segments $S_1,S_2,T_0$ and a proper convex subgroup $H$ of $G$. If $\ovl T
=\ovl T_0$ is the largest solution of $S_1/H+\ovl T=S_2/H$, then its preimage $T=
\varphi_H^{-1}(\ovl T_0)$ is the largest solution of $S_1+T=S_2\,$. In particular,
if $\ovl T=T_0/H$ is the largest solution of $S_1/H+\ovl T=S_2/H$, then its preimage
$T=T_0+H$ under $\varphi_H$ is the largest solution of $S_1+T=S_2\,$.
\end{lemma}

Take a final segment $T$ of $G$. We define 
\[
T\dhat
\]
to be the preimage of $\widehat{T/\cG(T)}$ under $\varphi_{\cG(T)}\,$. Note that
$T\dhat$ is a final segment of $G$. Observe that $T\dhat=T$ holds if and only if
$T/\cG(T)$ is closed in $G/\cG(T)$. If $\cG(T)$ is trivial, then $T\dhat=\widehat T$.
It follows from part 6) of Lemma~\ref{lemGS} that $T\dhat$ is closed also if $\cG(T)$ is
not trivial. We view $T\dhat$ as the ``deep closure'' of $T$.

\begin{lemma}                          \label{hatdhat}
Take a final segment $T$ of $G$.
\sn
1) The preimage of $-(T/\cG(T))^c$ under $\varphi_{\cG(T)}$ is $-T^c$.
\sn
2) For each $\alpha\in G$, the preimage of $(\alpha/\cG(T))^-$ under $\varphi_{\cG(T)}$ is
$\alpha^-+\cG(T)$.
\sn
3) If there is $\alpha\in G$ such that $\alpha/\cG(T)$ is the infimum of $T/\cG(T)$ in 
$G/\cG(T)$, then 
\[
T\dhat\>=\>\alpha^-+\cG(T)\>;
\]
otherwise, $T\dhat=T$. 
\end{lemma}
\begin{proof}
1): By part 1) of Lemma~\ref{cGcomp}, $\cG(-T^c)=\cG(T^c)=\cG(T)$, hence the preimage
of $-T^c/\cG(T)$ under $\varphi_{\cG(T)}$ is $-T^c+\cG(T)=-T^c$. It remains to
show that $-T^c/\cG(T)=-(T/\cG(T))^c$. Since $\cG(T^c)=\cG(T)\,$, the sets $T/\cG(T)$
and $T^c/\cG(T)$ are disjoint. On the other hand, their union is $G/\cG(T)\,$, so
$T^c/\cG(T)=(T/\cG(T))^c$, showing that $-T^c/\cG(T)=-(T/\cG(T))^c$ holds.

\sn
2): This follows from Lemma~\ref{M+H} since $\alpha^-/\cG(T)=(\alpha/\cG(T))^-$.

\sn
3): If there is $\alpha\in G$ such that $\alpha/\cG(T)$ is the infimum of $T/\cG(T)$ in 
$G/\cG(T)$, then $\widehat{T/\cG(T)}=(\alpha/\cG(T))^-$, whose preimage under
$\varphi_{\cG(T)}$ is $\alpha^-+\cG(T)$. Otherwise, $\widehat{T/\cG(T)}=T/\cG(T)$, whose
preimage is $T+\cG(T)=T$.
\end{proof}

\begin{proposition}                               \label{maxsol}
Take final segments $S_1$ and $S_2$ of $G$ and assume that $T=T_0$ is a solution 
of equation (\ref{eqT}). Set $T_1:=T_0+\cG(S_2)$. Then the following statements hold.
\sn
1) If $S_1/\cG(S_2)$ is principal, 
then $T=T_1$ is the largest solution of (\ref{eqT}).
\sn
2) If $S_1/\cG(S_2)$ is nonprincipal, 
then $T=T_1\dhat$ is the largest solution of (\ref{eqT}).
\end{proposition}
\begin{proof}
Since $S_1+T_1=S_1+T_0+\cG(S_2)=S_2+\cG(S_2)=S_2\,$, we have $S_1/\cG(S_2)+
T_1/\cG(S_2)=S_2/\cG(S_2)$. 
\sn
1): If $S_1/\cG(S_2)$ is principal, then by part 2) of Proposition~\ref{S_1pr}, $\ovl T=
T_1/\cG(S_2)$ is the unique solution of $S_1/\cG(S_2)+\ovl T=S_2/\cG(S_2)$. Hence by
Lemma~\ref{presol}, $T=T_1+\cG(S_2)=T_1$ is the largest solution of (\ref{eqT}).
\sn
2): We observe that $\cG(S_2/\cG(S_2))$ is trivial by part 2) of Proposition~\ref{invgr/H}.
By assumption, $S_1/\cG(S_2)$ is nonprincipal, so we can apply Lemma~\ref{partsol} to infer
that $\ovl T=\widehat{T_1/\cG(S_2)}$ is the largest solution of $S_1/\cG(S_2)+\ovl T=S_2/
\cG(S_2)$. Therefore, from its definition and Lemma~\ref{presol} we obtain that $T=
T_1\dhat$ is the largest solution of (\ref{eqT}).
\end{proof}

\begin{remark}
\rm Remark~\ref{cex} shows that $T=T_1\dhat$ is not always a solution of (\ref{eqT}). 
\end{remark}

For an arbitrary final segment $S$ of $G$, the following application of the
foregoing proposition gives an interesting characterization of $S\dhat$:
\begin{corollary}                              \label{chardhat}
Take a final segment $S$ of $G$. Then $T=S\dhat$ is the largest solution of the 
equation
\[
S\,+\,T\>=\>S\,+\,S\>. 
\]
\end{corollary}
\begin{proof}
We apply Proposition~\ref{maxsol} with $T_0=S$. Since $\cG(S+S)=\cG(S)$ by Corollary~\ref{cGnS} and $S+\cG(S)=S$, we obtain $T_1=S$. If $S/\cG(S)$ is nonprincipal,
then our claim follows
from part 2) of Proposition~\ref{maxsol}. Now assume that $S/\cG(S)$ is principal. Then 
by part 1) of Proposition~\ref{maxsol}, $S$ is the largest solution, and we need to show 
that $S=S\dhat$ in this case. As $S/\cG(S)$ is principal, it is closed in $G/\cG(S)$ and 
we indeed have $S\dhat=S$ by our observation before Lemma~\ref{hatdhat}.
\end{proof}

\parb
Now we answer questions (QFS1), (QFS2) and (QFS3) in full generality.

\begin{theorem}                             \label{solv}
Take final segments $S_1$ and $S_2$ of $G$. If condition 
\sn
$(*)$ \ $\cG(S_1)\subseteq \cG(S_2)$, and   
$S_2/\cG(S_2)$ is nonprincipal if $S_1/\cG(S_2)$ is nonprincipal
\sn
holds, then $T=S_2\ms S_1$ is the largest solution of equation (\ref{eqT}), and it is equal 
to: 
\sn
$\bullet$ $S_2-\alpha$ \ if \ $S_1/\cG(S_2)=(\alpha/\cG(S_2))^-$, 
\sn
$\bullet$ $(S_2-S_1^c)\dhat$ \ if \ $S_1/\cG(S_2)$ is nonprincipal.

\pars
If condition $(*)$ does not hold, or equivalently, one of the cases 
\sn
i) $\cG(S_2)\subsetneq\cG(S_1)$ or
\sn
ii) $\cG(S_1)= \cG(S_2)$ and $S_1/\cG(S_2)$ is nonprincipal, but $S_2/\cG(S_2)$
is principal
\sn
holds, then equation (\ref{eqT}) has no solution.
\end{theorem}
\begin{proof}
To simplify notation, we set $\cG_2:=\cG(S_2)$. Assume that condition $(*)$ holds. From part
2) of Proposition~\ref{invgr/H} we know that $\cG(S_1/\cG_2)=\cG(S_2/\cG_2)=\{0\}$. Hence 
by Proposition~\ref{solG=0}, the largest solution of $S_1/\cG_2 +\ovl T=S_2/\cG_2$ is
$\ovl T=S_2/\cG_2 - \alpha/\cG_2=(S_2-\alpha)/\cG_2$ if $S_1/\cG_2=(\alpha/\cG_2)^-$, and
$\ovl T=(S_2/\cG_2 - (S_1/\cG_2)^c)\,\widehat{ } \ $ if $S_1/\cG_2$ is nonprincipal.

By Lemma~\ref{presol} the preimage $T$ of $\ovl T$ under $\varphi_{\cG_2}$ is the largest
solution of (\ref{eqT}). As stated in the introduction, $S_2\ms
S_1$ is the largest solution of $S_1+T\subseteq S_2$, so it follows that $T=S_2\ms S_1\,$.

If $\ovl T=(S_2- \alpha)/\cG_2\,$, then $T= (S_2-\alpha)+\cG_2=S_2-\alpha$ by
Lemma~\ref{M+H} and part 3) of Lemma~\ref{GS+S}. Now assume that $S_1/\cG_2$ is
nonprincipal and $\ovl T=(S_2/\cG_2 - (S_1/\cG_2)^c)\,\widehat{ } \ $. By part 4) of
Proposition~\ref{invgr/H}, the former implies that $\cG_2\subseteq\cG(S_1)$, whence 
$\cG(S_1)=\cG_2\,$. Therefore, $\varphi_{\cG_2}^{-1}(- (S_1/\cG_2)^c)=-S_1^c$ by part 1)
of Lemma~\ref{hatdhat}. By Lemma~\ref{M+H}, $\varphi_{\cG_2}^{-1}(S_2/\cG_2)=S_2+\cG_2=
S_2\,$, hence $\varphi_{\cG_2}^{-1}(S_2/\cG_2 - (S_1/\cG_2)^c)=S_2-S_1^c$.
Thus the preimage of $\ovl T$ is $T=(S_2-S_1^c)\dhat$.

\parm
Now assume that condition $(*)$ does not hold. Then one of the cases i) or ii) must hold;
observe that if $\cG(S_1)\subsetneq\cG_2$, then by part 4) of Proposition~\ref{invgr/H},
$S_1/\cG_2$ is principal and condition $(*)$ holds.
Note that in both cases, $S_1$ is nonprincipal. In case i), 
Proposition~\ref{nosol} shows that equation (\ref{eqT}) has no solution. Assume that 
case ii) holds and that the final segment $T$ is a solution of equation (\ref{eqT}).
Then $S_1/\cG_2+T/\cG_2=S_2/\cG_2$, but this contradicts part 2) of
Lemma~\ref{sumfs} applied to the group $G/\cG_2\,$. 
\end{proof}

The following theorem addresses question (QFS4) in the case where no solution to  
equation (\ref{eqT}) exists: 
\begin{theorem}                               \label{maxsol3}
Take final segments $S_1$ and $S_2$ of $G$.
\sn
1) Assume that $\cG(S_2)\subsetneq\cG(S_1)$. Then there exists an element $\alpha\in S_2$ 
such that $S'_2:=(\alpha+\cG(S_1))^+\subsetneq S_2$ is the largest final segment contained 
in $S_2$ having invariance group $\cG(S_1)$.
\sn
2) Assume that $\cG(S_1)= \cG(S_2)$ and $S_1/\cG(S_1)$ is nonprincipal, but $S_2/\cG(S_1)$ 
is principal, say $S_2/\cG(S_1)=(\alpha/\cG(S_1))^-$ for some $\alpha\in S_2\,$. Then $S'_2
:=(\alpha+\cG(S_1))^+\subsetneq S_2$ is the largest final segment contained in $S_2$ such
that $S'_2/\cG(S_1)$ is nonprincipal.
\sn
In both cases, condition $(*)$ of Theorem~\ref{solv} is satisfied with $S'_2$ in place of
$S_2$ and $T=S_2'\ms S_1$ is the largest solution of $S_1+T=S_2'$. Further, $S_2'\ms S_1=
S_2\ms S_1$.
\end{theorem}
\begin{proof}
To simplify notation, we set $\cG_1:=\cG(S_1)$.
\sn
1): By the definition of $\cG(S_2)$, there is some $\alpha\in S_2$ such that $\alpha
+\cG_1$ is not contained in $S_2\,$. We set $S'_2=(\alpha+\cG_1)^+$. By definition, $S'_2$ 
is a final segment, and since $\alpha\in S_2\setminus S'_2$, we have that $S'_2\subsetneq 
S_2\,$. From part 5) of Lemma~\ref{GS+S} we know that $\cG(S'_2)=\cG_1\,$.

Suppose that there is a final segment $S$ of $S_2$ properly containing $S'_2$ and such
that $\cG_1\subseteq\cG(S)$. Then there would exist an element $\beta\in (\alpha+\cG_1)
\cap S_2$ such that $\beta+\cG_1\subseteq \beta+\cG(S) \subseteq S\subseteq S_2\,$.
However, it would follow that $\beta-\alpha\in \cG_1$ and $\alpha+\cG_1=\beta+\cG_1
\subseteq S_2\,$, contradiction. This shows that $S'_2$ is the largest final segment
contained in $S_2$ having an invariance group containing $\cG_1$.

Now if $S_1+T\subseteq S_2$, then $\cG_1\subseteq\cG(S_1+T)$ by part 2) of
Lemma~\ref{GS+S}, hence by what we have just shown, $S_1+T\subseteq S'_2\,$.

Assume that $S'_2/\cG_1=(\alpha/\cG_1)^+$ is principal; it then follows that $G/\cG_1$ is 
discretely ordered. Hence by part 5) of Proposition~\ref{invgr/H}, also $S_1/\cG_1$ is
principal. This shows that condition $(*)$ of Theorem~\ref{solv} holds with $S_2$ replaced
by $S_2'\,$. Now it follows from Theorem~\ref{solv} that $T=S_2'\ms S_1$ is the largest 
solution of $S_1+T=S_2'$. Since $S_2'\subset S_2$, we have $S_2'\ms S_1\subseteq S_2\ms 
S_1\,$. On the other hand, $\cG_1\subseteq \cG(S_1+(S_2\ms S_1))$ which implies that $S_1
+(S_2\ms S_1)\subseteq S_2'$ and therefore $S_2\ms S_1\subseteq S_2'\ms S_1\,$.

\sn
2): As in the proof of part 1) it is shown that $S'_2\subsetneq S_2$ and that $\cG(S'_2)
=\cG_1$. Also as in the proof of part 1) it is shown that if $S'_2/\cG_1=(\alpha/\cG_1)^+$
is principal, then so is $S_1/\cG_1$. But this contradicts our assumption, so
$S'_2/\cG_1$ is nonprincipal.


Any final segment $S$ that properly contains $S'_2=(\alpha+\cG_1)^+$ will contain an
element $\beta\in \alpha+\cG_1\,$. This implies that $S/\cG_1$ contains $\alpha/\cG_1$
and thus contains $(\alpha/\cG(S_1))^-=S_2/\cG_1\,$. Hence if $S$ is contained in $S_2\,$,
$S/\cG_1$ is contained in and therefore equal to $S_2/\cG_1\,$, which is principal.
This proves that $S'_2$ is the largest final segment contained in $S_2$ such that
$S'_2/\cG_1$ is nonprincipal.

Again, condition $(*)$ of Theorem~\ref{solv} holds with $S_2$ replaced by $S_2'$ so that
$T=S_2'\ms S_1$ is the largest solution of $S_1+T=S_2'$ and as before,
$S_2'\ms S_1\subseteq S_2\ms S_1\,$. To show the reverse inclusion, suppose that
$S_2'\subsetneq S_1+(S_2\ms S_1)$.
Then $(S_1+(S_2\ms S_1))/\cG_1$ contains $\alpha/\cG_1$ and thus $S_1+(S_2\ms S_1)$ must be 
equal to $S_2\,$, which is impossible by Theorem~\ref{solv}. Hence again, $S_1
+(S_2\ms S_1)\subseteq S_2'$ and therefore $S_2\ms S_1\subseteq S_2'\ms S_1\,$.
\end{proof}
 
\mn
\begin{remark}
\rm As mentioned already in the introduction, a survey and a list of references on the
arithmetic of cuts is given in \cite{Ku60}. However, there we missed (at least) one 
important reference, namely to the work of Schikhof, see e.g.\ \cite{OS}. The solution of
equations of the form $S_1+T=S_2$ is implicitly connected with the investigation of the set 
$S_1+G\fs$, which is done in \cite{OS}. There too, invariance groups, appearing under the 
name of {\it stabilizer}, play a crucial role, as well as the final segments $H^-$ and 
$H^+$. At this point, let us mention that in our terminology indicated in \cite{Ku60}, the 
cuts determined by $H^-$ and $H^+$ are called {\it group cuts}, and the {\it shifted} cuts 
determined by $\gamma+H^-$ and $\gamma+H^+$, with $\gamma\in G$, are called {\it ball cuts}.
\end{remark}

\bn 
%
%
%
\section{$\cO$-ideals}                   \label{sectideal}

For general background from valuation theory, we recommend \cite{EP,ZS}. 
Throughout, we consider a valued field $(K,v)$ with valuation ring $\cO_v$ and denote its
maximal ideal by $\cM_v\,$. Its value group will be denoted by $vK$, and its residue field
$\cO_v/\cM_v$ by $Kv$. For $a\in K$ we denote its value under $v$ by $va$.

Recall that by an $\cO_v$-ideal $I$ we mean a nonzero $\cO_v$-module $I\subsetneq K$.
We consider (possibly fractional) $\cO_v$-ideals $I$, 
$I_1\,$, $I_2\,$, $J$, etc. We are interested in the inequality 
\begin{equation}                   \label{ineqid}
I_1\,J\>\subseteq\> I_2
\end{equation}
and the corresponding equality (\ref{eqid}). We note that $I_2:I_1$ as defined in (\ref{:})
is an $\cO_v$-ideal, which may be fractional, and recall that $J=I_2:I_1$ is the largest 
$\cO_v$-ideal that satisfies (\ref{ineqid}). Further, we recall:
\begin{proposition}                               \label{vbij}
The map 
\[
v: I\>\mapsto\> vI\,=\,\{vb\mid 0\ne b\in I\} 
\]
establishes an order preserving bijection
between the $\cO_v$-ideals and the final segments of $vK$; its inverse map is 
\[
S\>\mapsto\>I_S\,=\,(a\in K\mid va\in S)\>. 
\]
\end{proposition}

The proofs of the following facts are straightforward.
\begin{lemma}                          \label{basic}
1) For $\cO_v$-ideals $I_1\,$, $I_2$ and $J$, the equation (\ref{eqid})             
holds if and only if
\begin{equation}                       \label{eqval}
vI_1\,+\, vJ\>=\>vI_2 \>.
\end{equation}
2) An $\cO_v$-ideal $I$ is principal if and only if the final segment $vI$ of $vK$ is 
principal.
\end{lemma}

%
We will also need:
\begin{lemma}              \label{I1:I2}
Take $\cO_v$-ideals $I_1$ and $I_2\,$. 
\sn
1) We have
\begin{equation}            \label{I1:I2eq1}
v(I_2:I_1)\>=\>\{\alpha\in vK \mid \alpha+vI_1\subseteq vI_2\}\>=\>vI_2\,\ms\,vI_1\>.
\end{equation}
In particular, 
\begin{equation}            \label{I1:I2eq2}
v(\cO_v:I_1) \>=\>\{\alpha\in vK \mid \alpha+vI_1\geq 0\}\>=\>\Delt vI_1\>. 
\end{equation}
and
\begin{equation}            \label{I1:I2eq3}
vI_2(\cO_v:I_1) \>=\> vI_2\,+\, \Delt vI_1 \>=\> vI_2\,\mfs\,vI_1\>.
\end{equation}
2) We have
\begin{equation}            \label{I1:I2eq4}
vI_2(\cO_v:I_1)\>=\>\left\{
\begin{array}{ll}
vI_2 -va & \mbox{if $I_1=a\cO_v\,$},\\
vI_2 - (vI_1)^c & \mbox{if $I_1$ is nonprincipal. }
\end{array}\right.
\end{equation}
\end{lemma}
\begin{proof}
1): \ Equation~(\ref{I1:I2eq1}) follows via Proposition~\ref{vbij} from the definitions of
$I_2:I_1$ and $vI_2\,\ms\,vI_1\,$. Equation~(\ref{I1:I2eq2}) follows from
equation~(\ref{I1:I2eq1}) since $\alpha+vI_1\geq 0\Leftrightarrow \alpha\geq -vI_1\,$.
Equation~(\ref{I1:I2eq3}) follows from equation~(\ref{I1:I2eq2}) and part 1) of
Lemma~\ref{+-}.
\sn
2): \ If $I_1=a\cO_v\,$, then $vI_1=(va)^-$ and $\Delt vI_1=(-va)^-$ by Lemma~\ref{Delt S},
hence by equations~(\ref{I1:I2eq3}) and (\ref{+alpha^-^+}), $vI_2(\cO_v:I_1)=vI_2+\Delt
vI_1=vI_2-va$. If $I_1$ is nonprincipal, then so is $vI_1$ and again by Lemma~\ref{Delt S},
$\Delt vI_1=-(vI_1)^c$, whence $vI_2(\cO_v:I_1)=vI_2 - (vI_1)^c$ by
equation~(\ref{I1:I2eq3}).
\end{proof}

\mn
%
%
%
\subsection{Overrings of $\cO_v$ and their ideals}   \label{sectprelid}
\mbox{ }\sn
Every overring of $\cO_v$ is again a valuation ring. Every such overring $\cO$ is obtained 
from $\cO_v$ by localization with respect to a prime ideal of $\cO_v$; this prime ideal 
then becomes the maximal ideal $\cM$ of $\cO$. The set $\cR$ of all valuation rings which
contain $\cO_v$ is linearly ordered by inclusion. Also the set of maximal ideals of these
valuation rings, i.e., the set of all prime ideals of $\cO_v\,$, is ordered by inclusion; if
$\cO_1,\cO_2\in\cR$ and $\cM_1,\cM_2$ are their respective maximal ideals, then $\cO_1
\subseteq\cO_2\Leftrightarrow\cM_2\subseteq\cM_1\,$. Note that for $\cO\in\cR$, an 
$\cO_v$-ideal $I$ is an $\cO$-ideal if and only if $I\cO=I$, and since all $\cO\in\cR$ 
are rings with 1, this holds if and only if $I\cO\subseteq I$.

For every $\cO\in \cR$, we set
\[
H(\cO)\>:=\> v\cO\cap -v\cO\>=\>v\cO^\times\>.
\]
This is a convex subgroup of the value group $vK$ of $(K,v)$. The valuation $w$ associated 
with $\cO$ is (up to equivalence) given by
\begin{equation}                 \label{wa}
wa\>=\> va/H(\cO) 
\end{equation}
for every $a\in K$, the value group of $w$ is canonically isomorphic to $vK/H(\cO)$, and 
the value group of the valuation induced by $v$ on the residue field $Kw$ is canonically 
isomorphic to $H(\cO)\,$ (cf.\ \cite{ZS}). 

Conversely, for every convex subgroup $H$ of $vK$, 
\begin{equation}                            \label{O(H)}
\cO(H)\>:=\>\{b\in K\mid \exists\alpha\in H:\>\alpha\leq vb\}\in\cR\>.
\end{equation}
We have that
\begin{equation}                            \label{vO(H)}
v\cO(H)\>=\>H^-\>.
\end{equation}
Note that
\begin{equation}                            \label{OHtimes}
v\cO(H)^\times \>=\> H\>,
\end{equation}
or in other words, $\cO(H)^\times= v^{-1}(H)$.
\pars
We recall the following facts from general valuation theory (cf.\ \cite{ZS}):
\begin{proposition}                               \label{invOH}
The map $H\mapsto\cO(H)$ is an order preserving bijection from
the set of all convex subgroups of $vK$ onto the set $\cR$ of all
valuation rings which contain $\cO_v\,$. Its inverse is the order
preserving map $\cR\ni \cO\mapsto H(\cO)$. Thus,
\begin{equation}
\cO(H(\cO))\>=\>\cO\mbox{ \ \ and \ \ } H(\cO(H))\>=\> H\>.
\end{equation}
\end{proposition}

We denote the maximal ideal of the valuation ring $\cO(H)$ by $\cM(H)$. It follows that 
\begin{equation}                        \label{M(H)}
\cM(H)\>=\>\{b\in K\mid vb>H\}\mbox{ \ \ and \ \ } v\cM(H)\>=\>H^+\>.
\end{equation}

\parm
We will read off information about an $\cO_v$-ideal $I$ from the
invariance group of the associated final segment $vI$. We start with the case
of $I=\cO\in \cR$. Note that for every $\cO\in \cR$ with maximal ideal $\cM$, we have that 
$\cO$ and $\cM$ are $\cO_v$-ideals.
\begin{lemma}                         \label{H(O)}
For every $\cO\in \cR$ with maximal ideal $\cM$,
\begin{equation}                            \label{H=IO=IM}
H(\cO)\>=\>\cG(v\cO)\>=\>\cG(v\cM)
\;\mbox{\ \ and\ \ }\; \cO\>=\>\cO(\cG(v\cO))
\>=\>\cO(\cG(v\cM))\>.
\end{equation}
Further, $v\cO$ is the disjoint union of $H(\cO)=v\cO^\times$ and $H(\cO)^+=v\cM$, with 
$v\cO^\times<v\cM$.
\end{lemma}
\begin{proof}
Denote the valuation associated with $\cO$ by $w$. For every $a\in \cO^\times$ and every 
$b\in\cM$, we have $wa=0<wb$. By the contrapositive of implication (\ref{presleq}), this 
implies that $H(\cO)=v\cO^\times<v\cM$, which shows that $v\cO$ is the disjoint union of
$v\cO^\times=H(\cO)$ and $v\cM$. Thus, $v\cO=H(\cO)^-$ and $v\cM=H(\cO)^+$.
Now (\ref{H=IO=IM}) follows from Proposition~\ref{invOH} and part 5) of Lemma~\ref{GS+S} 
where $H=H(\cO)$ and $\alpha=0$.
\end{proof}

\mn
%
%
%
\subsection{The invariance valuation ring of an $\cO_v$-ideal}         \label{sectivr}
\mbox{ }\sn
Take any $\cO_v$-ideal $I$. We set
\begin{equation}                            \label{O(I)}
\cO(I)\>:=\>\{b\in K\mid bI\subseteq I\}
\;\;\;\mbox{\ \ and\ \ }\;\;\;
\cM(I) \>=\>\{b\in K\mid bI\subsetneq I\}\>.
\end{equation}

\sn
For example, for every $\cO\in \cR$ with maximal ideal $\cM$,
\begin{equation}                            \label{O(I)ex}
\cO(\cO)\>=\>\cO(\cM)\>=\>\cO\;\;\;\mbox{\ \ and\ \ }\;\;\;
\cM(\cO)\>=\>\cM(\cM)\>=\>\cM\>.
\end{equation}
Further, we define 
\[
H(I)\>:=\>H(\cO(I))\>=\>v\cO(I)\cap -v\cO(I) \>=\> v\cO(I)^\times\>.
\]
\begin{theorem}                             \label{Ovr}
1) For every $\cO_v$-ideal $I$, $\,\cO(I)$ is a valuation ring of $K$ containing $\cO_v\,$, 
with maximal ideal $\cM(I)$, which is a prime $\cO_v$-ideal. It is the largest of all 
valuation rings $\cO$ of $K$ for which $I$ is an $\cO$-ideal. 
\sn
2) For every $\cO_v$-ideal $I$, $\cO(I)^\times=\{b\in K\mid bI=I\}$ and $H(I)$ is a convex
subgroup of $vK$.
\sn
3) For every $\cO_v$-ideal $I$,  
\begin{equation}                      \label{M(I)}
\cO(I)\>=\>\cO(\cG(vI))\;\mbox{ and }\; \cM(I)\>=\>\cM(\cG(vI))\>.
\end{equation}
as well as
\begin{equation}                      \label{H(I)}
H(I)\>=\>H(\cO(I))\>=\>\cG(vI)\>.
\end{equation}
\sn
4) For every $\cO_v$-ideal $I$,  
\begin{equation}                    \label{vM(I)}
\cM(I)\>=\>\{b\in K\mid vb>\cG(vI)\}\quad\mbox{ and }\quad v\cM(I)\>=\>\cG(vI)^+\>.
\end{equation} 
\end{theorem}
\begin{proof}
1): \ It is straightforward to prove that $\cO(I)$ is a ring. As
$I$ is an $\cO_v$-ideal, we have that $\cO_v I=I$. Hence,
$\cO_v\subseteq \cO(I)$. By general valuation theory, it
follows that $\cO(I)$ is a valuation ring.

We have $\{b\in K\mid bI=I\}\subseteq \cO(I)^\times$ since $bI=I\Leftrightarrow I=b^{-1}I$ 
for $b\ne 0$. The converse inclusion holds since if $b,b^{-1}\in \cO(I)$, then $I=bb^{-1}I
\subseteq bI \subseteq I$ and thus, $bI=I$. We have shown that $\cO(I)^\times=\{b\in K\mid
bI=I\}$. In every valuation ring, the unique maximal ideal consists precisely of all 
non-units. Hence the maximal ideal of $\cO(I)$ is $\cO(I)\setminus\{b\in K\mid bI=I\}
=\cM(I)$.

Finally, it follows directly from the definition that if $\cO$ is
a valuation ring of $K$ such that $I$ is an $\cO$-ideal, then
$\cO\subseteq \cO(I)$. On the other hand, it also follows
from the definition that $I$ is an $\cO(I)$-ideal; therefore,
$\cO(I)$ is the largest of all valuation rings of $K$ for which $I$ is an $\cO$-ideal.

\sn
2): \ The first assertion has already been shown in the proof of part 1). The second 
assertion follows from the definition of $H(I)$ and the fact that $\cO(I)$ is a valuation 
ring.

\sn
3): \ By part 6) of Lemma~\ref{GS+S}, we compute $\cO(I)=\{b\in K\mid bI\subseteq I\}=
\{b\in K\mid vb+vI\subseteq vI\}=\{b\in K\mid vb\in\cG(vI)^-\}=\cO(\cG(vI))$, and similarly, 
$\cM(I)=
\{b\in K\mid bI\subsetneq I\}=\{b\in K\mid vb+vI\subsetneq vI\}=\{b\in K\mid vb\in
\cG(vI)^+\}=\cM(\cG(vI))$. From what we have shown,
$H(I)=v\cO(I)^\times=v\{b\in K\mid bI= I\}=\cG(vI)^-\setminus\cG(vI)^+=\cG(vI)$.

\sn
4): \ Since $\cM(I)=\cM(\cG(vI))$, this follows from part 3) and (\ref{M(H)}).
\end{proof}

\begin{lemma}                               \label{HOM}
1) Take $\cO_v$-ideals $I$ and $J$. Then 
\[
\cO(IJ)\>=\>\cO(I)\cup\cO(J)\>=\>\max\{\cO(I)\,,\,\cO(J)\}\>.
\]
2) For every $\cO_v$-ideal $I$ and $n\in\N$, $\cO(I^n)=\cO(I)$.
\sn
3) For every $\cO_v$-ideal $I$ and $0\ne a\in K$, $\cO(aI)=\cO(I)$.
\end{lemma}
\begin{proof}
1): We have $\cG(vIJ)=\cG(vI+vJ)=\cG(vI)\cup\cG(vJ)$ by part 2) of Lemma~\ref{GS+S}. Our
statement follows from this by means of (\ref{M(I)}) and Proposition~\ref{invOH}.
\sn
2): This follows from part 2) by induction on $n$.
\sn
3): We have $\cG(vaI)=\cG(va+vI)=\cG(vI)$ by part 3) of Lemma~\ref{GS+S}. Our
statement follows from this by means of (\ref{M(I)}).
\end{proof}

\begin{lemma}                          \label{vaO(I)}
Take $\cO\in\cR$ and an $\cO_v$-ideal $I$.
\sn
1) We have that $I\cO$ is an $\cO$-ideal with
\begin{equation}                        \label{vIO}
vI\cO\>=\> vI\,+\, H(\cO)\>. 
\end{equation}
Consequently, $I$ is an $\cO$-ideal if and only if $H(\cO)\subseteq\cG(vI)$.
\sn
2) If $0\ne a\in K$, then 
\[
va\cO(I)\>=\> va^-+\cG(vI)\>. 
\]
3) Denote by $w$ the valuation associated with $\cO$. Then the preimage of $wI=vI/H(\cO)$
in $vK$ is equal to $vI\cO$.
\end{lemma}
\begin{proof}
1): We compute: $vI\cO=vI+v\cO=vI+(H(\cO)\cup H(\cO)^+)= (vI+H(\cO)) \cup (vI+H(\cO)^+)$, 
where we have used Lemma~\ref{H(O)}. Since $H(\cO)^+$ consists 
only of positive elements, we have $vI+H(\cO)^+\subseteq vI$. Since $0\in H(\cO)$, we have 
$vI\subseteq vI+H(\cO)$. Thus, $(vI+H(\cO)) \cup (vI+H(\cO)^+)=vI+H(\cO)$.

Further, $I$ is an $\cO$-ideal if and only if $I\cO=I$, which is equivalent to $vI+H(\cO)
=vI$. This in turn holds if and only if $H(\cO)\subseteq\cG(vI)$.

\sn
2): By equations (\ref{M(I)}) and (\ref{vO(H)}), $v\cO(I)=v\cO(\cG(vI))=\cG(vI)^-$. 
It follows that $va\cO(I)=va+v\cO(I)=va+0^-+\cG(vI)=va^-+\cG(vI)$, where we have used 
equation (\ref{0^-+M}) twice. 
\sn
3): By (\ref{wa}) and Lemma~\ref{M+H}, the preimage of $wI$ under $\varphi_{H(\cO)}$ is
$vI+H(\cO)$. By part 1) of our lemma, this is equal to $vI\cO$.
\end{proof}

\mn
%
%
%
\subsection{Properties of $\cM(I)$}               \label{sectM(I)}
\mbox{ }\sn
For $\cO_v$-ideals $I$, the ideals $\cM(I)$ appear in 
\cite[Chapter II, Section 4]{FS} under the notation $I^\sharp$. There, properties of
$\cM(I)=I^\sharp$ are listed. In what follows, we will present several of them and show 
how they follow from the work we have done so far, in particular from our results on 
final segments in ordered abelian groups in Section~\ref{fs}.

\begin{proposition}
Take $\cO_v$-ideals $I$ and $J$. Then the following statements hold:
\sn
1) $\cM(I)=\cM_v$ if $I$ is a principal ideal,
\sn
2) for $a\in K$, $\cM(aI)=\cM(I)$,
\sn
3) $\cO(I)\subseteq I$ if and only if $\cM(I)$ does not contain $I$,
\sn
4) $\cM(I)=I$ whenever $I$ is a prime $\cO_v$-ideal,
\sn
5) if $I$ and $J$ are prime $\cO_v$-ideals and $aI=J$ for some $a\in K$, then $I=J$,
\sn
6) $\cM(\cO_v:I)=\cM(I)$,
\sn
7) $I:\cM(I)=I$ if there is no $a\in K$ such that $\cM(I)=aI$; otherwise, $I:\cM(I)$ 
properly contains $I$,
\sn
8) $I\cM(I)\subsetneq I$ if and only if $I$ is a principal $\cO(I)$-ideal,
\sn
9) if $\cM(I)\subsetneq \cM(J)$, then $IJ=aI$ for some $a\in K$,
\sn
10) if $\cM(I)=\cM(J)$ and $J\cM(J)\subsetneq \cM(J)$, then $IJ=aI$ for some $a\in K$,
\sn
11) $\cM(I)$ is the union of all proper integral ideals $aI$, $a\in K$,
\sn
12) $I(\cO_v:I)=\cM(I)$ if $I$ is a nonprincipal $\cO_v$-ideal, and $I(\cO_v:I)
=\cO_v$ otherwise,
\sn
13) $\cM(IJ)=\cM(I)\cap\cM(J)=\min\{\cM(I)\,,\,\cM(J)\}$,
\sn
14) for each $n\in\N$, $\cM(I^n)=\cM(I)$.
\end{proposition}
\begin{proof}
Throughout, we will use (\ref{vM(I)}) and Proposition~\ref{vbij} freely.
\sn
1): If $I=a\cO_v$ with $a\in K$, then $vI=(va)^-$ and $\cG(vI)=\{0\}$ by (\ref{G(+-)}).
Hence $\cM(I)=\{b\in K\mid vb>0\}=\cM_v\,$.
\sn
2): Since $v(aI)=va+vI$, we have $\cG(v(aI))=\cG(vI)$ by part 3) of Lemma~\ref{GS+S}, whence 
$\cG(v(aI))^+=\cG(vI)^+$ and $\cM(aI)=\cM(I)$.
\sn
3): By Theorem~\ref{Ovr}, $I$ is an $\cO(I)$-ideal and $\cM(I)$ is the
maximal ideal of $\cO(I)$. Hence $\cO(I)\subseteq I$ if and only if $\cO(I)^\times\cap I$ 
is nonempty, and this holds if and only if $I$ is not contained in $\cO(I)\setminus
\cO(I)^\times=\cM(I)$.
\sn
4): If $I$ is a prime $\cO_v$-ideal, then it is the maximal ideal $\cM$ of a valuation ring 
$\cO\in\cR$ and $\cM(I)=\cM(\cM)=\cM=I$ by (\ref{O(I)ex}).
\sn
5): By parts 2) and 4), the assumptions imply $I=\cM(I)=\cM(aI)=\cM(J)=J$.
\sn
6): From (\ref{I1:I2eq2}) we know that $v(\cO_v:I)=\Delt vI$. By part 1) of
Lemma~\ref{cGcomp}, $\cG(\Delt vI)=\cG(vI)$. Using (\ref{M(I)}) twice, we obtain 
$\cM(\cO_v:I)=\cM(\cG(v(\cO_v:I)))=\cM(\cG(\Delt vI))=\cM(\cG(vI))=\cM(I)$.
\sn
7): Since $v(I:\cM(I))=\{va\mid va+v\cM(I)\subseteq vI\}$ by (\ref{I1:I2eq1}) of
Lemma~\ref{I1:I2} and $v\cM(I)=\cG(vI)^+$, the statement follows from 
part 2) of Proposition~\ref{S1S2}.
\sn
8): We have that $I\cM(I)\subsetneq I$ if and only if $vI+\cG(vI)^+\ne vI$. By
Lemma~\ref{S=cG(S)}, this holds if and only if $vI/\cG(vI)$ has a smallest element. Denote 
the valuation associated with $\cO(I)$ by $w$. Since $\cG(vI)=H(\cO(I))$ by (\ref{H(I)}),
this in turn holds if and only if $wI=vI/H(\cO(I))=vI/\cG(vI)$ has a
smallest element, which means that $I$ is a principal $\cO(I)$-ideal. 
\sn
9): The assumption $\cM(I)\subsetneq \cM(J)$ is equivalent to $\cO(J)\subsetneq \cO(I)$. 
By (\ref{M(I)}) and Proposition~\ref{invOH}, this is equivalent to $\cG(vJ)
\subsetneq \cG(vI)$. By part 3) of Proposition~\ref{S1S2}, there is $\alpha\in vK$ such 
that $vI+vJ=vI+\alpha$. Choosing $a\in K$ with $va=\alpha$, we obtain that $IJ=aI$.
\sn
10): If $\cM(I)=\cM(J)$, then $\cO(I)=\cO(J)$. Hence by part 8), $J\cM(J)\subsetneq\cM(J)$
implies that $J$ is a principal $\cO(J)$-ideal and hence a principal $\cO(I)$-ideal, say
$J=a\cO(I)$. Consequently, $IJ=Ia\cO(I)=aI\cO(I)=aI$.
\sn
11): Since $v\cM(I)=\cG(vI)^+$, this follows from part 5) of Proposition~\ref{S1S2}.
\sn
12): By (\ref{I1:I2eq3}), $vI(\cO_v:I)=vI\mfs vI$. If $vI$ is nonprincipal, then by 
part 1) of Corollary~\ref{corsols}, $vI\mfs vI=\cG(vI)^+$, so $I(\cO_v:I)=\cM(I)$.
If $vI$ is principal, then from (\ref{diffG}) with $\alpha=\beta$ we obtain $vI\mfs vI=
0^-$, so $I(\cO_v:I)=\cO_v\,$.
\sn
13) and 14) follow from parts 1) and 2) of Lemma~\ref{HOM}.

\end{proof}

\mn
%
%
%
\subsection{The equation $I_1J=I_2$}               \label{sectideq}
\mbox{ }\sn
We will use Propositions~\ref{invOH} and~\ref{vbij} to deduce the results in this section 
from our results in Section~\ref{fs}.
Our aim is to use these results to compute $I_2:I_1$ and to
determine under which conditions the equation $I_1 (I_2:I_1)=I_2$ holds. 

\pars
For every $\cO\in\cR$ with associated valuation $w$ and every $\cO$-ideal $J$, we 
define:
\[
\widehat J_{\cO}\>:=\>\left\{
\begin{array}{ll}
a\cO & \mbox{if there exists $a\in K$ such that $wa$ is the infimum of $wJ$, and} \\
J & \mbox{otherwise.}
\end{array}\right.
\]
This is well defined because if $a,b\in K$ with $wa=wb$, then $a\cO=b\cO$.
We will call $\widehat J_{\cO}$ the \bfind{closure of $J$ as an 
$\cO$-ideal}. If $wJ$ is closed in $wK$, then $\widehat J_{\cO}=J$.

\begin{proposition}
For every $\cO$-ideal $J$, 
\[
\cO:(\cO:J)\>=\> \widehat J_{\cO}\>. 
\]
\end{proposition}
\begin{proof}
Define $w$ as above. Applying (\ref{I1:I2eq2}) of
Lemma~\ref{I1:I2} twice, we obtain $w(\cO:(\cO:J))=\Delt w(\cO:J)=\Delt(\Delt wJ)$. By
part 3) of Lemma~\ref{Delt S}, this is equal to the closure of $wJ$ in $wK$. Our assertion
now follows by Proposition~\ref{vbij}.
\end{proof}

\pars
For every $\cO_v$-ideal $J$, we define:
\[
J\dhat \>:=\> \widehat J_{\cO(J)}\>.      
\]
We have the following characterizations of $J\dhat$:
\begin{proposition}                   \label{Jdhat}
Take any $\cO_v$-ideal $J$. 
\sn
1) We have
\[
J\dhat \>=\> I_{(vJ)\dhat}\>.
\]
\sn
2) We have
\[
J^2:J\>=\> J\dhat\>.
\]
Hence if $vJ/\cG(vJ)$ is closed in $vK/\cG(vJ)$, then $(J^2:J)=J$.
\end{proposition}
\begin{proof}
1): By Theorem~\ref{Ovr}, $J$ is an $\cO(J)$-ideal. Denote by $w$ the valuation associated 
with $\cO(J)$, so that $wa=va/H(\cO(J))$ for every $a\in K$. By (\ref{H(I)}), 
$H(\cO(J))=\cG(vJ)$. By definition, $(vJ)\dhat$ is the preimage of $\widehat{vJ/\cG(vJ)}=
\widehat{vJ/H(\cO(J))}=\widehat{wJ}$ under $\varphi_{H(\cO(J))}$. If $wJ$ is closed in
$wK$, then $\widehat{wJ}=wJ=vJ/\cG(vJ)$ whose preimage is $vJ+\cG(vJ)=vJ$ and therefore,
$I_{(vJ)\dhat}=I_{vJ}=J=J\dhat$. Now assume that $wJ$ is not closed in $wK$ and take
$a\in K$ such that $wa$ is the infimum of $wJ$ in $wK$. Then $\widehat{wJ}=(wa)^-$ whose
preimage is $va^-+\cG(vJ)$. By part 2) of Lemma~\ref{vaO(I)}, this is equal to $va\cO(I)$.
It follows that $I_{(vJ)\dhat}=I_{vaO(I)}=aO(I)=J\dhat$. This proves our assertion.

\sn
2): This follows from part 1) and Corollary~\ref{chardhat}.
\end{proof}

\parm
We now consider solutions $J$ of equation (\ref{eqid}). 
\begin{proposition}                               \label{idmaxsol}
Take $\cO_v$-ideals $I_1$ and $I_2$ and assume that $J=J_0$ is a solution 
of equation (\ref{eqid}). Set $J_1:=J_0\cO(I_2)$. Then the following statements hold.
\sn
1) If $I_1\cO(I_2)$ is a principal $\cO(I_2)$-ideal, then $I_2:I_1=J_1$.
\sn
2) If $I_1\cO(I_2)$ is a nonprincipal $\cO(I_2)$-ideal, then $I_2:I_1=J_1\dhat$. 
\end{proposition}
\begin{proof}
Denote by $w$ the valuation associated with $\cO(I_2)$. We have $H(\cO(I_2))=\cG(vI_2)$ by
(\ref{H(I)}). From Lemma~\ref{H(O)} it follows that $H(\cO(I_2))=v\cO(I_2)^\times$ and 
therefore, $v\cO(I_2)/H(\cO(I_2))=(0/H(\cO(I_2))^-$. Hence, 
\begin{eqnarray*}
wI_1\cO(I_2)&=&v(I_1\cO(I_2))/H(\cO(I_2))\>=\>vI_1/H(\cO(I_2))+v\cO(I_2)/H(\cO(I_2))\\
&=& vI_1/H(\cO(I_2))+(0/H(\cO(I_2))^-\>=\>vI_1/H(\cO(I_2))\>=\>vI_1/\cG(vI_2)\>. 
\end{eqnarray*}
By part 2) of Lemma~\ref{basic}, $I_1\cO(I_2)$ is a principal $\cO(I_2)$-ideal if and only 
if $wI_1\cO(I_2)$ is principal. Further, using (\ref{vIO}), we find that
$vJ_0\cO(I_2)=vJ_0+H(\cO(I_2))=vJ_0+\cG(vI_2)$.
Now statements 1) and 2) follow from part 1) of Proposition~\ref{Jdhat} together with  
Proposition~\ref{maxsol} where $S_1=vI_1\,$, $S_2=vI_2\,$, and $T_0=vJ_0\,$.
\end{proof}

\pars
From Theorem~\ref{solv} we deduce:
\begin{theorem}                             \label{idsolv2}
Take $\cO_v$-ideals $I_1$ and $I_2\,$. If condition 
\sn
$(*)$ \ $\cO(I_1)\subseteq \cO(I_2)$, and  
$I_2$ is a nonprincipal $\cO(I_2)$-ideal if $I_1$ is a nonprincipal $\cO(I_2)$-ideal
\sn
holds, then $J=I_2:I_1$ is the largest solution of equation (\ref{eqid}), and it is equal
to:
\sn
$\bullet$ \ $a^{-1}I_2$ \ if $I_1\cO(I_2)=a\cO(I_2)$,
\sn
$\bullet$ \ $(I_2(\cO_v:I_1))\dhat$ \ if $I_1\cO(I_2)$ is a nonprincipal $\cO(I_2)$-ideal.

\pars
If condition $(*)$ does not hold, or equivalently, one of the cases  
\sn
i) $\cO(I_2)\subsetneq\cO(I_1)$ or
\sn
ii) $\cO(I_1)= \cO(I_2)$ and $I_1$ is a nonprincipal $\cO(I_1)$-ideal, but $I_2$
is a principal\par $\ \cO(I_1)$-ideal
\sn
holds, then equation (\ref{eqid}) has no solution.
\end{theorem}
\begin{proof}
Since $\cO(I)=\cO(\cG(vI))$ for every $\cO_v$-ideal $I$ by equation (\ref{M(I)}), we know 
that $\cO(I_1)\subseteq \cO(I_2)$ if and only if $\cG(vI_1)\subseteq \cG(vI_2)$. Part 2) 
of Lemma~\ref{basic}, with the valuation $w$ associated with $\cO(I_2)$ in place of $v$,
shows for $i=1,2$ that $I_i$ is a principal $\cO(I_2)$-ideal if and only if $wI_i=
vI_i/\cG(I_2)$ is principal. Further, if $vI_1/\cG(I_2)$ is nonprincipal, then also
$I_1$ is nonprincipal and therefore, $vI_2(\cO_v:I_1)=vI_2 - (vI_1)^c$ by (\ref{I1:I2eq4}).

By part 1) of Lemma~\ref{basic}, $I_1 J=I_2$ holds if and only if $vI_1 +vJ=vI_2\,$. 
Now all of our assertions follow from Theorem~\ref{solv}, where $S_i=vI_i$ for $i=1,2$.
%
\end{proof}

The following theorem addresses question (QID4) in the case where no solution to  
equation (\ref{eqid}) exists. 
\begin{theorem}                               \label{idsolv3}
Take $\cO$-ideals $I_1$ and $I_2\,$. 
\sn
i) Assume that $\cO(I_2)\subsetneq\cO(I_1)$. Then there exists $a\in I_2$ such that 
$I'_2:=a\cM(I_1)\subsetneq I_2$ is the largest $\cO(I_1)$-ideal contained in $I_2\,$. 
\sn
ii) Assume that $\cO(I_1)= \cO(I_2)$ and $I_1$ is a nonprincipal $\cO(I_1)$-ideal, but 
$I_2$ is a principal $\cO(I_1)$-ideal, say $I_2=a\cO(I_1)$ for some $a\in K$. In this case,
$I'_2:=a\cM(I_1)\subsetneq I_2$ is the largest nonprincipal $\cO(I_1)$-ideal contained in
$I_2$. 
\sn
In both cases, $J=I'_2:I_1$ is the largest solution of $I_1\,J= I'_2\,$,
and $I'_2:I_1=I_2:I_1\,$.
\end{theorem}
\begin{proof}
This is a consequence of Theorems~\ref{maxsol3} and~\ref{idsolv2}. We just need to observe 
that $(va+\cG(vI_1))^+=va+\cG(vI_1)^+$, whence $I_{(va+\cG(vI_1))^+}=aI_{\cG(vI_1)^+}=
a\cM(I_1)$, where the last equality holds since $\cG(vI_1)^+=v\cM(I_1)$ by (\ref{vM(I)}).
\end{proof}

\mn
%
%
%
\subsection{Annihilators of quotients $I_1/I_2$}               \label{sectann}
\mbox{ }\sn
Take $\cO_v$-ideals $I_1$ and $I_2\,$. The \bfind{annihilator} of the quotient $I_1/I_2$
with $I_2\subsetneq I_1$ is the $\cO_v$-ideal $I_2:I_1$. In general, the equation $I_1 J=
I_2$ is not solvable and $J=I_2:I_1$ will not be a solution, but it will be the largest 
$\cO_v$-ideal such that $I_1 J\subseteq I_2\,$.
However, in the situation appearing in the papers \cite{pr1,pr2}, we have $I_1=U$ and 
$I_2=UV$ for $\cO_v$-ideals $U,V$. Then $J=V$ is a solution of $I_1 J=I_2$ and our task is 
just to compute the largest solution $I_2:I_1$, which we accomplish in the next theorem.
\begin{theorem}                              \label{ann1}
Take $\cO_v$-ideals $I$ and $J$. Then 
\[
I J:I=\>\left\{
\begin{array}{ll}
J\cO(I J) & \mbox{if $I\cO(I J)$ is a principal $\cO(I J)$-ideal, and} \\
(J\cO(I J))\dhat & \mbox{if $I\cO(I J)$ is a nonprincipal $\cO(I J)$-ideal.}
\end{array}\right.
\] 
In the special case where $\cO(I)=\cO(J)$, we have $\cO(IJ)=\cO(I)$, $I\cO(I J)=I$ and 
$J\cO(I J)=J$, so that $IJ:I=J$ if $I$ is a principal $\cO(I)$-ideal, and $IJ:I=
J\dhat$ otherwise.
\end{theorem}
\begin{proof}
Our assertions follow from Proposition~\ref{idmaxsol} with $I_1=I$, $I_2=IJ$ and $J_0=J$. 
If $\cO(I)=\cO(J)$, 
then $\cO(I J)=\cO(I)\cup\cO(J)=\cO(I)=\cO(J)$ by part 1) of Lemma~\ref{HOM}, 
which yields our last assertion.
\end{proof}

\begin{lemma}                              \label{annlem}
Take $\cO_v$-ideals $I_1$ and $I_2$ with $I_2\subsetneq I_1\,$. If the annihilator of
$I_1/I_2$ is $\cM_v\,$, then $I_1$ is a principal $\cO_v$-ideal.
\end{lemma}
\begin{proof}
If $I_1$ is nonprincipal, then $vI_1$ is nonprincipal and by part 3) of Lemma~\ref{+0^+}
we obtain $vI_1=vI_1+0^+=vI_1+v\cM_v\,$, whence $I_1=I_1\cM_v\,$. As $\cM_v$ annihilates
$I_1/I_2$, it follows that $I_1=I_1\cM_v\subseteq I_2\,$, which contradicts our assumption.
\end{proof}

\begin{proposition}                              \label{ann2}
1) Take $\cO_v$-ideals $I_1$ and $I_2$ with $I_2\subsetneq I_1\,$. Then the annihilator
of $I_1/I_2$ is equal to $\cM_v$ if and only if $I_1$ is principal and $I_2=I_1\cM_v\,$. 
\sn
2) Take $\cO_v$-ideals $I$ and $J$ with $J\subsetneq\cO_v\,$. Then the annihilator of
$I/I J$ is equal to $\cM_v$ if and only if $I$ is principal and $J=\cM_v\,$.
\end{proposition}
\begin{proof}
1) Assume that $I_1$ is principal and $I_2=I_1\cM_v\,$. Then $\cM_v$ annihilates $I_1/I_2$. 
On the other hand, $I_2=I_1\cM_v\subsetneq I_1$, so $I_1/I_2$ is not zero and $\cO_v$ does 
not annihilate it. Thus, $\cM_v$ is the annihilator of $I_1/I_2\,$.

In order to prove the converse, assume that $\cM_v$ is the annihilator of
$I_1/I_2\,$. From Lemma~\ref{annlem} we infer that $I_1$ must be principal, say $I_1=(a)$ 
for some $a\in K$. As $\cM_v$ annihilates $I_1/I_2\,$, we have $a\cM_v=I_1\cM_v\subseteq
I_2\,$. We cannot have $a\cM_v\subsetneq I_2$ since then $I_1=(a)\subseteq I_2\,$, which
contradicts our assumption. Thus we have $I_2=a\cM_v=I_1\cM_v\,$. 
\sn
2): By part 1), applied to $I_1=I$ and $I_2=I J$, the annihilator of $I/I J$ is equal to 
$\cM_v$ if and only if $I$ is principal and $I J=I\cM_v\,$. As $I$ is principal, the latter 
is equivalent to $J=\cM_v\,$.
\end{proof}

Finally, we treat the general case. 
\begin{theorem}                                \label{anngen}
Take $\cO_v$-ideals $I_1,I_2$ with $I_2\subsetneq I_1\,$.
\sn
1) Assume that (\ref{eqid}) has a solution. Then the annihilator of $I_1/I_2$ is
$a^{-1}I_2$ if $I_1\cO(I_2)=a\cO(I_2)$, and is 
$I_2(\cO_v:I_1)\dhat$ if $I_1\cO(I_2)$ is a nonprincipal $\cO(I_2)$-ideal. 
\sn
2) Assume that (\ref{eqid}) does not have a solution, and let $I'_2$ be as defined in 
Theorem~\ref{idsolv3}. Then the annihilator of $I_1/I_2$ is 
$a^{-1}I'_2$ if $I_1\cO(I_2)=a\cO(I_2)$, and is
$I'_2(\cO_v:I_1)\dhat$ if $I_1\cO(I_2)$ is a nonprincipal $\cO(I_2)$-ideal. 
\end{theorem}
\begin{proof}
1): This follows from Theorem~\ref{idsolv2}.
\sn
2): This follows from Theorems~\ref{idsolv3} and~\ref{idsolv2}.
\end{proof}

\mn
%
%
%
\subsection{Application to the computation of annihilators in special cases} 
\label{sectannsp}
\mbox{ }\sn
For an $\cO_v$-module $M$ we denote by $\ann\,M$ the annihilator of $M$.
\begin{proposition}
Take an $\cO_v$-ideal $I\subsetneq\cO_v$ and $n\in\N$, and set $J=bI^{n-1}$, where
$b\in\cO_v\,$.
If $n=1$, assume that $b\in\cM_v\,$. Then $I/IJ=I/bI^n$ is not zero, and the following statements hold.
\sn
1) If there is $a\in K$ such that $vJ/\cG(vJ)$ has infimum $va/\cG(vJ)$ in 
$vK/\cG(vJ)$ but does not contain this infimum, then 
\begin{equation}              \label{annform}
\ann I/IJ\>=\>J\dhat\>=\>a\cO(J) \>,
\end{equation}
which properly contains $J$. In all other cases, $\ann I/IJ=J$. 

In all cases, $\cM(J)\,\ann I/IJ\subseteq J$. 

\sn
2) We have that $\ann I/IJ=\cM_v$ if and only if $\cM_v$ is principal and either
$I=J=\cM_v$ (which holds if and only if $n=2$ and $b\in \cO_v^\times$), or $I=\cO_v$ and
$J=\cM_v=(b)$ (which holds if and only if $n=1$ and $b\notin \cO_v^\times$).
\end{proposition}
\begin{proof}
Our assumption that $b\in\cM_v$ in the case of $n=1$ and that $b\in\cO_v$ and
$I\subsetneq\cO_v$ in the case of $n>1$ implies that $J\subsetneq\cO_v\,$. This yields that
$I/IJ$ is not zero.

Since $vJ=vb+(n-1)vI$, part 3) of Lemma~\ref{GS+S} together with Corollary~\ref{cGnS} 
shows that $\cG(vI)=\cG(vJ)$, which implies that $\cO(I)=\cO(J)$.
\sn
1): From Theorem~\ref{ann1} we infer that $\ann I/IJ=J$ if $I$ is a principal $\cO(J)$-ideal,
and $\ann I/IJ=J\dhat$ if $I$ is a nonprincipal $\cO(J)$-ideal. Assume first that there is
$a\in K$ such that $vJ/\cG(vJ)$ has infimum $va/\cG(vJ)$ in $vK/\cG(vJ)$ but does not 
contain this infimum. Then $vJ/\cG(vJ)$ is nonprincipal, and the same must hold for 
$vI/\cG(vJ)$, so that $I$ is a nonprincipal $\cO(J)$-ideal. It follows that $\ann I/IJ
=J\dhat=a\cO(J)$. Now assume the remaining case, that is, $vJ/\cG(vJ)$ is closed in 
$G/\cG(vJ)$ and therefore, $J\dhat=J$. In this case we obtain that $\ann I/IJ=J$.

Finally, we prove the last assertion. Since $\cM(J)\subseteq \cM_v\,$, it is trivial when
$\ann I/IJ=J$. In the remaining case, $\ann I/IJ=a\cO(J)$, whence $\cM(J)\,\ann I/IJ=
a\cM(J)$. Since $va/\cG(vJ)=wa$ is the infimum of $vJ/\cG(vJ)=wJ$ in $vK/\cG(vJ)=wK$ and
$wc>0$ for every $c\in\cM(J)$ and therefore $wac>wa$, it follows that $wa\cM(J)\subseteq
wJ$. Hence $w\cM(J)\,\ann I/IJ\subseteq wJ$, so $\cM(J)\,\ann I/IJ\subseteq J$.

\sn
2): By part 2) of Proposition~\ref{ann2}, the annihilator of $I/I J$ is equal to $\cM_v$ 
if and only if $I$ is principal and $J=\cM_v\,$. Assume first that $I$ is principal and
$J=\cM_v\,$.
Since $J=bI^{n-1}$, it follows that $J$ and $\cM_v$ are principal. Since $I$ is principal, 
$vI$ is principal and $\cG(vJ)=\cG(vI)=\{0\}$, so $vJ/\cG(vI)=vJ$ is principal and from 
part 1) we obtain that $\cM_v=\ann I/IJ=J=bI^{n-1}\subseteq I$. Thus we either have that
$I=\cM_v=J$, or that $I=\cO_v$ and $J=\cM_v=(b)$. Since $\cM_v$ is principal, the
former holds if and only if $n=2$ and $b\in \cO_v^\times$, and the latter holds 
if and only if $n=1$ and $b\notin \cO_v^\times$.

Now assume that $\cM_v$ is principal and either $I=\cM_v=J$ or $I=\cO_v$ and $J=\cM_v=(b)$.
If the former holds and $\cM_v=(c)$ for some $c\in \cM_v\,$, then $I/IJ=(c)/(c^2)$ and 
$\ann I/IJ=(c)=\cM_v\,$. If the latter holds, then $I/IJ=\cO/(b)$ and $\ann I/IJ=(b)
=\cM_v\,$. 
\end{proof}

The following facts are used in \cite{pr2}.
\begin{lemma}                     \label{aM^n}
Take $n\geq 2$, any valuation ring $\cO$ with maximal ideal $\cM$, and $a\in\cO$. Then
\sn
1) \ $a\cM=\cM$ if and only if $a\notin\cM$, 
\sn
2) \ $\cM^n=\cM$ if and only if $\cM$ is a nonprincipal $\cO$-ideal,
\sn
3) \ $(a\cM)^n= a\cM$ if and only if $a\notin\cM$ and $\cM$ is a 
nonprincipal $\cO$-ideal.
\end{lemma}
\begin{proof}
1): We have $a\notin a\cM$, hence if $a\in\cM$, then $a\cM\ne\cM$.
If $a\notin\cM$, then $a$ is a unit in $\cO$, so $a\cM=\cM$.
\sn
2): The value group of $w$ is not discrete, and hence dense, if and only if $\cM$ is a
nonprincipal $\cO$-ideal. If it is discrete and $\gamma$ is its smallest positive element,
then $\cM=\{b\in K\mid wb\geq\gamma\}$ and $\cM^n=\{c\in K\mid wc\geq n\gamma\}\subsetneq
\cM$ since $n\gamma>\gamma$. If it is dense, then for every $b\in\cM$ there is $c\in K$ 
such that $0<nwc<wb$, whence $b\in \cM^n$; therefore, $\cM^n\subseteq\cM\subseteq\cM^n$ and
consequently, $\cM^n=\cM$.
\sn
3): If $a\notin\cM$ and $\cM$ is a nonprincipal $\cO$-ideal, then by parts 1) and 2), 
$(a\cM)^n= \cM^n=\cM=a\cM$.
If $a\in\cM$, then $wa>0$, whence $a\cM=\{c\in K\mid wc>wa\}$ and $(a\cM)^n\subseteq 
a^n\cM=\{c\in K\mid wc>nwa\}\subsetneq a\cM$ since $nwa>wa$. If $\cM$ is a principal 
$\cO$-ideal, say $\cM=b\cO$ with $b\in\cM$, then $a\cM=ab\cO=\{c\in K\mid wc\geq wab\}$ 
and $(a\cM)^n=(ab\cO)^n=\{c\in K\mid wc\geq nwab\}\subsetneq a\cM$ since $nwab>wab$ and 
$ab\in a\cM$.
\end{proof}

\begin{proposition}
Take $n\geq 2$, $a\in\cO_v$ and $\cO\in\cR$ with maximal ideal $\cM$. Assume that
$(a\cM)^n\ne a\cM$. Then the following statements hold.
\sn
1) We have that
\[
\ann a\cM/(a\cM)^n \>=\> \left\{
\begin{array}{ll}
(a\cM)^{n-1} & \mbox{if $\cM$ is a principal $\cO$-ideal, } \\
(a\cO)^{n-1} = a^{n-1}\cO & \mbox{if $\cM$ is a nonprincipal $\cO$-ideal.}
\end{array}\right.
\]
2) \ The annihilator is equal to $\cM_v$ if and only if $n=2$, $a\notin\cM_v$ and
$\cM=\cM_v$ is a principal $\cO_v$-ideal.
\end{proposition}
\begin{proof}
1): We apply Theorem~\ref{ann1} with $I=a\cM$ and $J=(a\cM)^{n-1}$. Then $\cO(IJ)=
\cO((a\cM)^n)=\cO(a\cM)=\cO(\cM)=\cO$ by parts 2) and 3) of Lemma~\ref{HOM} and 
(\ref{O(I)ex}). Since $I$ and $J$ are already $\cO$-ideals, we obtain that $I\cO(IJ)=I$
and $J\cO(IJ)=J$. We observe that $I=a\cM$ is a principal $\cO$-ideal if and only the same 
is true for $\cM$.

It remains to show that $J\cO(IJ)\dhat=J\dhat$ is equal to $(a\cO)^{n-1}$ if $\cM$ is a
nonprincipal $\cO$-ideal. Denote the valuation associated with $\cO$ by $w$.
In the present case, $J=(a\cM)^{n-1}=a^{n-1}\cM$, and $0$ is the
infimum of $w\cM$, not contained in $w\cM$. Therefore,
$wa^{n-1}\notin wa^{n-1}+w\cM=wa^{n-1}\cM$ is the infimum of $wa^{n-1}\cM$.
Hence by definition, $J\dhat=\{c\mid wc\geq wa^{n-1}\}=a^{n-1}\cO=(a\cO)^{n-1}$. This
completes the proof of part 1).
\sn
2): Assume first that $\ann a\cM/(a\cM)^n=\cM_v\,$. Then $a\cM/(a\cM)^n$ is nontrivial,
whence $a\cM \ne (a\cM)^n$, which by part 3) of Lemma~\ref{aM^n}
means that $a\in\cM$ or $\cM$ is a principal $\cO$-ideal. If the latter holds, then by 
part 1), $a\cM:(a\cM)^n=(a\cM)^{n-1}$. Since $a\cM\subseteq\cM_v\,$, this is equal
to $\cM_v$ if and only if $n-1=1$, $\cM=\cM_v\,$, and $va=0$, where we have used
parts 1) and 3) of Lemma~\ref{aM^n}. On the other hand, if the latter does not hold, then
$a\in\cM$ and by part 1), $\ann a\cM/(a\cM)^n=(a\cO)^{n-1}\subseteq a\cO$. The latter is
contained in $\cM$ but is not equal to $\cM$ since $\cM$ is a nonprincipal $\cO$-ideal.
Thus $\ann a\cM/(a\cM)^n=a\cM:(a\cM)^n\subsetneq\cM\subseteq\cM_v$ and consequently,
$\ann a\cM/(a\cM)^n\ne\cM_v\,$.

For the converse, assume that $n=2$, $a\notin\cM_v$ and $\cM=\cM_v$ is a principal
$\cO_v$-ideal. Then by part 1) together with part 1) of Lemma~\ref{aM^n},
$\ann a\cM/(a\cM)^n=\ann a\cM_v/(a\cM_v)^n=(a\cM_v)^{n-1}=a\cM_v=\cM_v\,$.
\end{proof}



\begin{thebibliography}{99}

\bibitem{BFS} 
Bazzoni, S.\ -- Fuchs, L.\ -- Salce, L.: {\it The hierarchy of uniserial modules over a
valuation domain}, Forum Math.\ {\bf 7} (1995), 247--277
 
\bibitem{BS} 
Bazzoni, S.\ -- Salce, L.: {\it Groups in the class semigroups of valuation domains}, 
Israel J.\ Math.\ {\bf 95} (1996), 135--155
 
\bibitem{pr1} 
Cutkosky, S. D.\ -- Kuhlmann, F.-V.\ -- Rzepka, A.: {\it On the computation of K\"ahler
differentials and characterizations of Galois extensions with independent defect}, 
Mathematische Nachrichten {\bf 298} (2025), 1549--1577

\bibitem{pr2} 
Cutkosky, S. D.\ -- Kuhlmann, F.-V.: {\it K\"ahler differentials of extensions of
valuation rings and deeply ramified fields}, submitted; https://arxiv.org/abs/2306.04967

\bibitem{EP} {Engler, A.J.\ -- Prestel, A.: {\it Valued fields},
Springer Monographs in Mathematics. Springer-Verlag, Berlin, 2005}

\bibitem{FM} 
Fornasiero, A.\ -- Mamino, M.: {\it Arithmetic of Dedekind cuts of ordered abelian groups}, 
Ann.\ Pure Appl.\ Logic {\bf 156} (2008), 210--244

\bibitem{FS} 
Fuchs, L.\ -- Salce, L.: {\it Modules over non-Noetherian domains}, Mathematical Surveys and
Monographs {\bf 84}. American Mathematical Society, Providence, RI, 2001

\bibitem{Ku60} Kuhlmann, F.-V.: {\it Selected methods for the classification of cuts 
and their applications}, Proceedings of the ALANT 5 conference 2018, Banach Center
Publications {\bf 121} (2020), 85--106

\bibitem{KKF} 
Kuhlmann, F.-V.\ - Kuhlmann, S.\ - Fornasiero, A.: {\it Towers of complements to valuation
rings and truncation closed embeddings of valued fields}, Journal of Algebra {\bf 323} 
(2010), 574--600

\bibitem{pr3} 
Kuhlmann, F.-V.: {\it Invariance group and invariance valuation ring of a cut}, manuscript
available at https://www.fvkuhlmann.de/CUTS.pdf

\bibitem{OS} 
Olivos, E.\ -- Schikhof, W.~H.: {\it  Extending the multiplication of a totally ordered group 
to its completion}, Advances in non-Archimedean analysis, 231–242, Contemp.\ Math.\ {\bf 551},
Amer.\ Math.\ Soc., Providence, RI, 2011

\bibitem{R} Ribenboim, P.: {\it Sur les groupes totalement ordonn\'es et l'arithm\'etique des
anneaux de valuation}, Summa Brasil.\ Math.\ {\bf 4} (1958), 1--64.

\bibitem{ZS} Zariski, O.\ -- Samuel, P.: {\it Commutative Algebra}, Vol.\ II, 
New York--Heidelberg--Berlin, 1960

\end{thebibliography}
\end{document}